\documentclass[12pt, twoside, a4paper]{amsart}
\usepackage{amsmath, amsthm, amsopn, amssymb,
a4wide,
varioref, amscd}

\usepackage{color}
\usepackage{mathtools}
\usepackage{bm}
\usepackage{hyperref}
\usepackage{tikz-cd}
\usepackage{stmaryrd }
\newcommand{\comment}[1]{}

\theoremstyle{theorem}
    \newtheorem{theorem}{Theorem}
    \newtheorem{lemma}[theorem]{Lemma}
    
    \newtheorem{corollary}[theorem]{Corollary}

    \newtheorem{definition}[theorem]{Definition}
    
\theoremstyle{definition}

    \newtheorem{remark}[theorem]{Remark}
    \newtheorem{example}[theorem]{Example}
    \newtheorem{exercise}[theorem]{Exercise}

    \newtheorem{lem}[theorem]{Lemma}

\newcommand\mnote[1]{} 
\newcommand\be{\begin{equation*}}

\newcommand\ee{\end{equation*}}

\newcommand\ben{\begin{equation}}
\newcommand\een{\end{equation}}
\newcommand\bes{\begin{eqnarray*}}
\newcommand\ees{\end{eqnarray*}}

\newcommand\bex{\begin{exercise}}
\newcommand\eex{\end{exercise}}
\newcommand\beg{\begin{example}}
\newcommand\eeg{\end{example}}
\newcommand\benu{\begin{enumerate}}
\newcommand\eenu{\end{enumerate}}
\newcommand\beit{\begin{itemize}}
\newcommand\eeit{\end{itemize}}
\newcommand\berk{\begin{remark}}
\newcommand\eerk{\end{remark}}
\newcommand\bdefn{\begin{defintion}}
\newcommand\edefn{\end{definition}}
\newcommand\bthm{\begin{theorem}}
\newcommand\ethm{\end{theorem}}
\newcommand\bprf{\begin{proof}}
\newcommand\eprf{\end{proof}}
\newcommand\blem{\begin{lemma}}
\newcommand\elem{\end{lemma}}

\newcommand{\sm}{{\raise0.3ex\hbox{$\scriptstyle \setminus$}}}

\def\CHI{\mathchoice%
{\raise2pt\hbox{$\chi$}}%
{\raise2pt\hbox{$\chi$}}%
{\raise1.3pt\hbox{$\scriptstyle\chi$}}%
{\raise0.8pt\hbox{$\scriptscriptstyle\chi$}}}
\def\smalloplus{\raise1pt\hbox{$\,\scriptstyle \oplus\;$}}

\numberwithin{equation}{section}

\usepackage{enumerate}
\usepackage{color}

\usepackage{amsaddr}

\begin{document}
\title[And\^o dilations]{And\^{o} dilations for a pair of commuting contractions: two explicit constructions and functional models}
\author[Sau]{Haripada Sau}
\address{Department of Mathematics, Virginia Tech, Blacksburg, VA 24061-0123, USA\\ sau@vt.edu, haripadasau215@gmail.com}
\address{\tiny Dedicated to Professor Joseph A. Ball, a leading operator theorist, on the occasion of his 70th birthday. }
\thanks{This research is supported by a Post Doctoral Fellowship at Indian Institute of Technology Bombay and SERB Indo-US Postdoctoral Research Fellowship, 2017.}
\begin{abstract}
One of the most important results in operator theory is And\^o's \cite{ando} generalization of dilation theory for a single contraction to a pair of commuting contractions acting on a Hilbert space. While there are two explicit constructions (Sch\"affer \cite{sfr} and Douglas \cite{Doug-Dilation}) of the minimal isometric dilation of a single contraction, there was no such explicit construction of an And\^o dilation for a commuting pair $(T_1,T_2)$ of contractions, except in some special cases \cite{A-M-Dist-Var, D-S, D-S-S}. In this paper, we give two new proofs of And\^o's dilation theorem by giving both Sch\"affer-type and Douglas-type explicit constructions of an And\^o dilation with function-theoretic interpretation, for the general case. The results, in particular, give a complete description of all possible factorizations of a given contraction $T$ into the product of two commuting contractions. Unlike the one-variable case, two minimal And\^o dilations need not be unitarily equivalent. However, we show that the compressions of the two And\^o dilations constructed in this paper to the minimal dilation spaces of the contraction $T_1T_2$, are unitarily equivalent.

In the special case when the product $T=T_1T_2$ is pure, i.e., if $T^{* n}\to 0$ strongly, an And\^o dilation was constructed recently in \cite{D-S-S}, which, as this paper will show, is a corollary to the Douglas-type construction. We also show that their construction in this special case can be derived from a previous result obtained in \cite{sau}.

We define a notion of characteristic triple for a pair of commuting contractions and a notion of coincidence for such triples. We prove that two pairs of commuting contractions with their products being pure contractions are unitarily equivalent if and only if their characteristic triples coincide. We also characterize triples which qualify as the characteristic triple for some pair $(T_1,T_2)$ of commuting contractions such that $T_1T_2$ is a pure contraction.

\end{abstract}
\maketitle

\section{Introduction}
A result by Sz.-Nagy \cite{sz-nagy} that has influenced the development of operator theory greatly is that for every contraction $T$ acting on a Hilbert space $\mathcal H$, there exists an isometry $V$ acting on a Hilbert space $\mathcal K$ containing $\mathcal H$ such that $V^*|_{\mathcal H}=T^*$. A decade later, And\^o in his remarkable paper \cite{ando} extended this classical result of Sz.-Nagy to two variables by giving an enigmatic construction of a pair of commuting isometries $(V_1,V_2)$ for a pair of commuting contractions $(T_1,T_2)$ such that $(V_1,V_2)$ is a co-extension of $(T_1,T_2)$. Before we proceed further, we define the central topic of this paper.
\begin{definition}
Let $\underline{T}=(T_1,T_2,\dots,T_n)$ be a commuting $n$-tuple of operators on a Hilbert space $\mathcal H$. An $n$-tuple of commuting operators $\underline{V}=(V_1,V_2,\dots,V_n)$ on a Hilbert space $\mathcal K$ is called a {\it{dilation}} of $\underline{T}$, if there exists an isometry $\Pi:\mathcal{H}\to\mathcal{K}$ such that
$$
\Pi T_i^*=V_i^*\Pi, \text{ for every } i=1,2,\dots,n.
$$Moreover, the dilation $(\Pi,\underline{V})$  of $\underline{T}$ is said to be {\it{minimal}} if
\begin{eqnarray}\label{minimal-dilation}
\mathcal K=\overline{\text{span}}\{V_1^{m_1}V_2^{m_2}\cdots V_n^{m_n}\Pi h: m_i \geq 0 \text{ for each } 1\leq i \leq n,\; h\in \mathcal H\}.
\end{eqnarray}
\end{definition}When it is clear from the context what the isometry $\Pi$ is, we omit it. And\^o's theorem sparked a great deal of research in Mathematics, see \cite{Ando-triple, Archer, BBM, Cole-Wermer, Gadidov, Li-Ti, Parrott, Popescu} and references therein. However, an explicit construction of And\^o dilation with function theoretic interpretation has been lacking. And, only under substantial assumptions on the pair $(T_1,T_2)$, an And\^o dilation was constructed in the papers \cite{A-M-Dist-Var}, \cite{D-S} and \cite{D-S-S}. On the other hand, there are two concrete constructions \cite{sfr,Doug-Dilation} of the minimal isometric dilation of a single contraction. We shall recall both of these constructions here and give two-variable analogues of these classical constructions of dilation. In other words, we give two new proofs of And\^o's dilation theorem.

Note that if $\underline{V}=(V_1,V_2)$ on $\mathcal{K}$ is an And\^o dilation of a pair $(T_1,T_2)$ of commuting contractions on $\mathcal{H}$, then $V_1V_2$ is an isometric dilation of $T_1T_2$ and in general, $V_1V_2$ need not be the minimal isometric dilation of $T_1T_2$. In other words, $\mathcal{K}$, in general, is bigger than
$$
\mathcal{K}^{\text{min}}_{\underline{V}}:=\overline{\text {span}}\{V_1^mV_2^m\Pi h:h\in\mathcal{H}\text{ and }m\geq 0 \}.
$$While any two minimal isometric dilations of a single contraction are unitarily equivalent \cite{Nagy-Foias}, minimality in several variables does not yield uniqueness up to unitary equivalence. However, we prove that for a given pair $(T_1,T_2)$ of commuting contractions on a Hilbert space $\mathcal{H}$, the two And\^o dilations $\underline{V}=(V_1,V_2)$ and $\underline{V'}=(V_1', V_2')$ constructed here are such that when compressed to the respective minimal spaces  $\mathcal{K}^{\text{min}}_{\underline{V}}$ and $\mathcal{K}^{\text{min}}_{\underline{V'}}$, they are unitarily equivalent, i.e.,
$$
P_{\mathcal{K}^{\text{min}}_{\underline{V}}}(V_1,V_2)|_{\mathcal{K}^{\text{min}}_{\underline{V}}} \text{ is unitarily equivalent to }P_{\mathcal{K}^{\text{min}}_{\underline{V'}}}(V_1',V'_2)|_{\mathcal{K}^{\text{min}}_{\underline{V'}}},
$$where $P_{\mathcal{K}^{\text{min}}_{\underline{V}}}$ and $P_{\mathcal{K}^{\text{min}}_{\underline{V'}}}$ denote the projections onto $\mathcal{K}^{\text{min}}_{\underline{V}}$ and $\mathcal{K}^{\text{min}}_{\underline{V'}}$, respectively.

\subsection{Sch\"affer model for And\^o dilation}
A couple of years after Sz.-Nagy proved his dilation theorem, Sch$\ddot{\text a}$ffer in \cite{sfr} gave the first explicit construction of a minimal isometric dilation of a contraction. An interesting application of this concrete construction is that it gives a constructive proof the famous commutant lifting theorem, see \cite{DMP} . Sch\"affer showed that if $T$ is a contraction on a Hilbert space $\mathcal H$, then the operator $V_S:\mathcal H\oplus H^2(\mathcal D_T)\to\mathcal H\oplus H^2(\mathcal D_T)$ given by
\begin{eqnarray}\label{Schef}
V_S:=\left(
     \begin{array}{cc}
      T & 0 \\
      D_T & M_z \\
     \end{array}
   \right),
\end{eqnarray}is an isometry (obviously a dilation of $T$). Here, for a contraction $T$, the space $\mathcal D_T$ is the closure of the range of the defect operator $D_T:=(I-T^*T)^{1/2}$ of $T$. The operator $M_z$ is the `forward shift' on $H^2(\mathcal{D}_T)$. The operator in the $(2,1)$-entry of the matrix in (\ref{Schef}) should be viewed as the constant function $z\mapsto D_Th$ in $H^2(\mathcal D_T)$, when applied to an element $h$ of $\mathcal H$ (this convention is taken up throughout the paper). For a contraction $T$, the notation $V_S$ in this paper will always denote the matrix in (\ref{Schef}). For a Hilbert space $\mathcal F$, the notation $H^2(\mathcal F)$ denotes the Hilbert space consisting of $\mathcal F$-valued analytic functions on the unit disk $\mathbb D$ for which the coefficients (belonging to $\mathcal F$) of its Taylor series expansion around the origin, are norm-square summable. Note that $H^2\otimes \mathcal F$ is another realization of $H^2(\mathcal F)$, where $H^2$ is the Hardy space over the unit disk. For $\varphi$ in $H^\infty(\mathcal B(\mathcal F))$, the algebra of $\mathcal B(\mathcal F)$-valued bounded analytic functions on $\mathbb D$, let $M_\varphi$ denote the bounded operator on $H^2(\mathcal F)$ defined by
$$
M_\varphi f(z)=\varphi(z)f(z), \text{ for all }f\in H^2(\mathcal F)\text{ and }z\in \mathbb D.
$$
Our first construction of And\^o dilation is Sch\"affer-type, which we now describe. See \S \ref{FutureResearch} for a possible application of this construction.

Let $(T_1,T_2)$ be a pair of commuting contractions and $T=T_1T_2$. We show that the space on which the dilation pair $(V_1,V_2)$ acts, can be chosen to be of the form $\mathcal K_S:=\mathcal H\oplus H^2(\mathcal F)$ for some Hilbert space $\mathcal F$ containing an isometric copy of $\mathcal{D}_T$, where
\begin{enumerate}
\item[(i)]$\mathcal F=\mathcal D_{T_1}\oplus\mathcal D_{T_2}$, if $\text{dim}(\mathcal D_{T_1}\oplus\mathcal D_{T_2})<\infty$ and
\item[(ii)]$\mathcal F=\mathcal D_{T_1}\oplus\big(\mathcal D_{T_2}\oplus l^2\big)$, (possibly) if $\text{dim}(\mathcal D_{T_1}\oplus\mathcal D_{T_2})=\infty$.
\end{enumerate} And\^o's construction was an influential result at the time but has some disadvantages: it does not lead to an explicit identification of a minimal dilation $(V_1,V_2)$ nor to any function-theoretic interpretation. However, we show that the dilation pair can be constructed in a way to have the following interesting structure:
$$
(V_1,V_2)|_{H^2(\mathcal F)}=(M_{P^\perp U+zP U},M_{U^*P+zU^*P^\perp}),
$$for some unitary $U$ and projection $P$ in $\mathcal B(\mathcal F)$.

Moreover, this construction leads to a minimal dilation in the following sense weaker than (\ref{minimal-dilation}). We find an isometry $\Lambda:\mathcal D_T\to \mathcal F$ and show that
\begin{eqnarray}\label{minimality}
\Pi_\Lambda^*V_1V_2\Pi_\Lambda=V_S,
 \end{eqnarray}where $\Pi_\Lambda:\mathcal{H}\oplus(H^2\otimes\mathcal{D}_T)\to\mathcal{H}\oplus(H^2\otimes\mathcal{F})$ is the isometry defined by
\begin{eqnarray}\label{Pi-Lambda}
\Pi_\Lambda:=I_{\mathcal H}\oplus \left(I_{H^2}\otimes \Lambda \right).
\end{eqnarray}
In $2\times 2$ block operator matrix representation with respect to the decomposition $\mathcal{H}\oplus H^2(\mathcal{F})$, the Sch\"affer-type dilation pair $(V^S_1,V^S_2)$ is the following:
$$
\left(\left(
     \begin{array}{cc}
      T_1 & 0 \\
      PU\Lambda D_T & M_{P^\perp U+zPU} \\
     \end{array}
   \right),\left(
     \begin{array}{cc}
     T_2 & 0 \\
      U^*P^\perp \Lambda D_T &  M_{U^*P+zU^*P^\perp } \\
     \end{array}
   \right)\right).
$$
This is the content of the following theorem  -- the first main result of this paper, which in particular describes all possible factorizations of a given contraction into the product of two commuting contractions.
\begin{theorem}[Sch\"affer model]\label{themainresult}
Let $(T_1,T_2,T)$ be a triple of contraction operators on a Hilbert space $\mathcal H$. Then the following are equivalent:
\begin{enumerate}
\item[({\bf P})] $(T_1,T_2)$ is commuting and $T$ is the product of $T_1$ and $T_2$;
\item[({\bf A})] There exists a Hilbert space $\mathcal F$, a commuting pair $(V^S_1,V^S_2)$ of isometries on $\mathcal{K}_S=\mathcal H\oplus H^2(\mathcal F)$ and an isometry $\Lambda:\mathcal D_T \to \mathcal F$ such that
$$
(V_1^{S *},V_2^{S *})|_{\mathcal H}=(T_1^*,T_2^*),\; (V_1^S,V_2^S)|_{H^2(\mathcal F)}=(M_\varphi,M_\psi) \text{ and } \;\Pi_\Lambda^*V_1^SV_2^S\Pi_\Lambda=V_S
$$where $\Pi_\Lambda:\mathcal H\oplus H^2(\mathcal D_T)\to\mathcal H\oplus H^2(\mathcal F)$ is the isometry as in (\ref{Pi-Lambda}) and $$\varphi(z)=(P^\perp U+zP U), \; \psi(z)=U^*P+zU^*P^\perp$$ are inner functions for some unitary $U$ and projection $P$ in $\mathcal B(\mathcal F)$;
\item[({\bf S})]There exists a pair of contractions $(S_1,S_2)$ on $\mathcal H\oplus H^2(\mathcal D_T)$ such that
\begin{eqnarray*}
(S_1^*,S_2^*)|_{\mathcal H}=(T_1^*,T_2^*),\;V_S^*|_{\mathcal H}=S_1^*S_2^*|_{\mathcal H}=S_2^*S_1^*|_{\mathcal H}\text{ and } \\(S_1,S_2)|_{H^2(\mathcal D_T)}=(M_\Phi,M_\Psi),
\end{eqnarray*}where $\Phi$ and $\Psi$ are some one-degree contractive $\mathcal B({{\mathcal D_T}})$-valued polynomials.

Moreover, $\Phi$ and $\Psi$ can be chosen to be $$\Phi(z)=F_1+zF_2^* \text{ and } \Psi(z)=F_2+zF_1^*$$ where $(F_1,F_2)=(\Lambda^*P^\perp U\Lambda,\Lambda^*U^*P\Lambda)$, the isometry $\Lambda$, the projection $P$ and the unitary $U$ are as in part (\bf{A}).
\end{enumerate}
\end{theorem}
Note that ({\bf P})$\Rightarrow$({\bf A}) of the above theorem gives an And\^o dilation for the pair $(T_1,T_2)$ while ({\bf A})$\Rightarrow$({\bf S}) describes the avatar of the dilation pair on the Sch$\ddot{\text a}$ffer dilation space.

\subsection{Douglas model for And\^o dilation}
There is another elegant construction of a minimal isometric dilation of a contraction by R. G. Douglas, see \S 4 of \cite{Doug-Dilation}. We describe it below. For a contraction $T$ acting on a Hilbert space $\mathcal{H}$, one always has
$$I_{\mathcal{H}}\geq TT^*\geq T^2{T^*}^2\geq\cdots\geq T^n{T^*}^n\geq\cdots\geq 0,$$ which implies that there exists a positive operator $Q$ such that $Q^2:=SOT\lim T^n{T^*}^n$. An immediate observation one makes about $Q$ is that $TQ^2T^*=Q^2$, which indicates that there exists an isometry $X^*$ from $\overline{Ran Q}$ into itself such that
\begin{eqnarray}\label{TheXIntro}
X^*Q=QT^*.
\end{eqnarray} Let $W^*$ on $\mathcal{R}\supseteq \overline{Ran Q}$ be the minimal unitary extension of $X^*$.
Define the operator $\mathcal{O}:\mathcal{H}\to H^2(\mathcal{D}_{T^*})$ by
\begin{eqnarray}\label{the-iso}
\mathcal O(h):=\sum_{n=0}^{\infty} z^n D_{T^*}{T^*}^nh \text{ for all $h \in \mathcal{H}$}.
\end{eqnarray}The operator $\mathcal O$ is called the {\em{observability operator}}, see \cite{BBF}. It was observed in \cite{Doug-Dilation} that the isometry $\Pi_{D}:\mathcal{H}\to H^2(\mathcal D_{T^*})\oplus \mathcal{R}$ defined by
$$
\Pi_{D}( h):= \mathcal O (h)\oplus Q(h),
$$has the following intertwining property
\begin{eqnarray}\label{The-NF}
\Pi_{D}T^*=(M_z\oplus W)^*\Pi_{D}.
\end{eqnarray}
Therefore the operator $V_{D}:=M_z\oplus W$ on $H^2(\mathcal D_{T^*})\oplus \mathcal{R}$ is an isometric dilation of $T$. It is shown to be minimal too, see Lemma 1 in  \cite{Doug-Dilation}.

As before, let $(T_1,T_2)$ be a pair of commuting contractions on a Hilbert space $\mathcal{H}$ and $T=T_1T_2$. We show that an And\^o dilation of $(T_1,T_2)$ can also be constructed on a space of the form $\mathcal K_{D}:=H^2(\mathcal{F}_*)\oplus \mathcal{R}$, where $\mathcal{F}_*$ is a Hilbert space containing an isometric image of $\mathcal{D}_{T^*}$ and like in the first construction
\begin{enumerate}
\item[(i)]$\mathcal{F}_*=\mathcal D_{T_1^*}\oplus\mathcal D_{T_2^*}$, if $\text{dim}(\mathcal D_{T_1^*}\oplus\mathcal D_{T_2^*})<\infty$ and
\item[(ii)]$\mathcal{F}_*=\mathcal D_{T_1^*}\oplus\big(\mathcal D_{T_2^*}\oplus l^2\big)$, (possibly) if $\text{dim}(\mathcal D_{T_1^*}\oplus\mathcal D_{T_2^*})=\infty$.
\end{enumerate}We find two commuting unitaries $W_1,W_2$ acting on $\mathcal{R}$ such that $\overline{RanQ}$ is co-invariant under $W_1$ and $W_2$ and that $W_1W_2=W$. We find an isometry $\Gamma:\mathcal{D}_{T^*}\to \mathcal{F}_*$ such that the isometry $\tilde\Pi:\mathcal{H}\to H^2(\mathcal{F}_*)\oplus \mathcal{R}$ defined by
$$
 \tilde\Pi(h):=\big((I_{H^2}\otimes \Gamma)\oplus I_{\mathcal{R}}\big)\Pi_{D}(h)=\sum_{n=0}^{\infty}z^n \Gamma D_{T^*}T^{* n}h\oplus Qh
$$has the following intertwining property:
$$
 \tilde\Pi(T_1,T_2,T_1T_2)^*=\big((M_{U^{' *}P'^\perp+zU^{' *}P'}\oplus W_1),(M_{P'U'+zP'^\perp U'}\oplus W_2), (M_z\oplus W)\big)^* \tilde\Pi,
$$where $P'$ and $U'$ are a projection and a unitary in $\mathcal{B}(\mathcal{F}_*)$. Consequently, the following pair of block operator matrices on $H^2(\mathcal{F}_*)\oplus \mathcal{R}$ is an And\^o dilation for the pair $(T_1,T_2)$:
$$
(V^D_1,V^D_2):=\left(\left(
     \begin{array}{cc}
      M_{U^{' *}P'^\perp+zU^{' *}P'} & 0 \\
      0 & W_1 \\
     \end{array}
   \right),\left(
     \begin{array}{cc}
      M_{P'U'+zP'^\perp U'} & 0 \\
      0 & W_2 \\
     \end{array}
   \right)\right).
$$
Note that the dilation $(V^D_1,V^D_2)$ is such that with the isometry $\Pi_\Gamma: (H^2\otimes \mathcal{D}_{T^*})\oplus\mathcal{R}\to (H^2\otimes\mathcal{F}_*)\oplus \mathcal{R}$ defined by
\begin{eqnarray}\label{Pi-Gamma}
\Pi_\Gamma:=(I_{H^2}\otimes \Gamma)\oplus I_{\mathcal{R}}
\end{eqnarray}the following holds:
\begin{eqnarray}\label{NF-min}
  \Pi_\Gamma^*V^D_1V^D_2\Pi_\Gamma &=& M_z\oplus W =V_{D}.
\end{eqnarray}
The following theorem, the second main result of the paper, summarizes the second construction.
\begin{theorem}[Douglas model]\label{the2ndmainresult}
Let $(T_1,T_2,T)$ be a triple of contraction operators on a Hilbert space $\mathcal H$. Then the following are equivalent:
\begin{enumerate}
\item[({\bf P})] $(T_1,T_2)$ is commuting and $T$ is the product of $T_1$ and $T_2$;
\item[({\bf A$'$})] There exist Hilbert spaces $\mathcal{F}_*$, $\mathcal{R}$, commuting unitaries $W_1,W_2$ in $\mathcal B(\mathcal R)$, a pair $(V^D_1,V^D_2)$ of commuting isometries on $H^2(\mathcal{F}_*)\oplus \mathcal{R}$ such that
    $$(V^D_1,V^D_2)=\left((M_{U^{' *}P'^\perp+zU^{' *}P'}\oplus W_1),(M_{P'U'+zP'^\perp U'}\oplus W_2)\right)$$ for some projection $P'$ and unitary $U'$ in $\mathcal{B}(\mathcal{F}_*)$, and a joint $(V^D_1,V^D_2)$-co-invariant subspace $\mathcal{M}\subseteq H^2(\mathcal{F}_*)\oplus\mathcal{R}$ such that
$(T_1,T_2,T)$ is unitarily equivalent to
$$
P_{\mathcal{M}}(V^D_1,V^D_2,V^D_1V^D_2)|_{\mathcal{M}}.
$$
Moreover, the space $\mathcal{M}$ can be chosen to be the range of an isometry $\tilde \Pi:\mathcal{H}\to H^2(\mathcal{F}_*)\oplus \mathcal{R}$ such that
  $$
  (V^D_1,V^D_2,V^D_1V^D_2)^*\tilde \Pi=\tilde \Pi(T_1,T_2,T_1T_2)^*;
  $$
\item[({\bf D})]There exist a pair of contractions $(D_1,D_2)$ acting on $H^2(\mathcal D_{T^*})\oplus \mathcal{R}$ and a joint $(D_1,D_2,V_D)$-co-invariant subspace $\mathcal{M'}\subseteq H^2(\mathcal D_{T^*})\oplus\mathcal{R}$ such that
\begin{eqnarray*}
(T_1,T_2, T)\text{ is unitarily equivalent to }P_{\mathcal{M'}}(D_1,D_2,V_D)|_{\mathcal M'} \text{ and }\\
P_{\mathcal{M'}}V_D|_{\mathcal{M'}}=P_{\mathcal{M'}}D_1D_2|_{\mathcal{M'}}=P_{\mathcal{M'}}D_2D_1|_{\mathcal{M'}}.
\end{eqnarray*}

Moreover, $(D_1,D_2)$ can be chosen to be $(M_\Phi\oplus W_1,M_{\Psi}\oplus W_2)$, where $$\Phi(z)=G_1^*+zG_2 \text{ and } \Psi(z)=G_2^*+zG_1,$$ where $(G_1,G_2)=(\Gamma^*P'^\perp U'\Gamma,\Gamma^*U'^*P'\Gamma)$, $\Gamma:\mathcal D_{T^*}\to \mathcal{F}_*$ is an isometry, $W_1$, $W_2$, $P'$ and $U'$ are as in part (\bf{A$'$}).
\end{enumerate}
\end{theorem}
Note that ({\bf P$'$})$\Rightarrow$({\bf A$'$}) of the above theorem gives an And\^o dilation for the pair $(T_1,T_2)$ while ({\bf A$'$})$\Rightarrow$({\bf D}) describes the avatar of the dilation pair on the Douglas dilation space.

Equations (\ref{minimality}) and (\ref{NF-min}) imply that the two And\^o dilations constructed above are such that when we compress their products $V_1V_2$ and $V^D_1V^D_2$ to the Sch\"affer's space $\mathcal{H}\oplus H^2(\mathcal{D}_T)$ and the Douglas' space $H^2(\mathcal{D}_{T^*})\oplus\mathcal{R}$, respectively, we get the minimal isometric dilations $V_S$ and $V_D$ of the product $T_1T_2$ of the contractions constructed by Sch\"affer and Douglas, respectively. We use the fact that any two minimal isometric dilations of a given contraction, in particular $V_S$ and $V_D$ are unitarily equivalent, to obtain the following result. In fact, in \S \ref{uniquenessSec} we prove a more general result (see Theorem \ref{uniqueness} below) from which the following theorem will follow.
\begin{theorem}\label{compression-uniquenss}
For a pair $(T_1,T_2)$ of commuting contractions on a Hilbert space $\mathcal{H}$, let $(V^S_1,V^S_2)$ on $\mathcal{K}_S=\mathcal{H}\oplus H^2(\mathcal F)$ and $(V^D_1,V^D_2)$ on $\mathcal{K}_D=H^2(\mathcal{F}_*)\oplus \mathcal{R}$ be the And\^o dilations constructed in Theorem \ref{themainresult} and Theorem \ref{the2ndmainresult}, respectively. Let $\Pi_{\Lambda}$ and $\Pi_\Gamma$ be the isometries as defined in (\ref{Pi-Lambda}) and (\ref{Pi-Gamma}), respectively. Then
$$
\Pi_{\Lambda}^*(V^S_1,V^S_2,V^S_1V^S_2)\Pi_{\Lambda}\text{ is unitarily equivalent to }\Pi_\Gamma^*(V^D_1,V^D_2,V^D_1V^D_2)\Pi_\Gamma.
$$
\end{theorem}
Although a triple of commuting contractions does not dilate, in general \cite{Parrott}, we observe that the triple $(T_1,T_2,T_1T_2)$ of commuting contractions always dilates to the triple $(V_1,V_2,V_1V_2)$ of commuting isometries, where $(V_1,V_2)$ is an And\^o dilation of $(T_1,T_2)$ and that a simple application of And\^o's theorem yields: \textit{$T=T_1T_2$ with $T_1,T_2$ commuting contractions if and only if there exists a commuting pair $(V_1,V_2)$ of isometries such that $(V_1,V_2,V_1V_2)$ is a co-extension of $(T_1,T_2,T)$.}

\subsection{Functional model}An important tool in dilation theory and functional model theory is the {\em{characteristic function}}, which for a contraction $T$, is defined by
\begin{eqnarray}\label{char}
\Theta_T(z):=[-T+zD_{T^*}(I_\mathcal{H}-zT^*)^{-1}D_T]|_{\mathcal{D}_T}, \text{ for all $z \in \mathbb{D}$}.
\end{eqnarray}This, at first glance intimidating, expression of the characteristic function actually is an obvious generalization of the M$\ddot{\text o}$bius transformations preserving $\mathbb D$, when one considers the scalar contractions. By virtue of the relation $TD_T=D_{T^*}T$ (see Chapter I of \cite{Nagy-Foias}), it follows that for each $z$ in $\mathbb{D}$, $\Theta_T(z)$ is in $\mathcal B(\mathcal D_T,\mathcal D_{T^*})$. At this point, we define a couple of terminologies concerning the characteristic function, the first one is due to Sz.-Nagy and Foias.
\begin{definition}
Let $T$ and $T'$ be contractions acting on Hilbert spaces $\mathcal{H}$ and $\mathcal{H'}$ respectively.
\begin{enumerate}
\item[(I)]The characteristic functions of $T$ and $T'$ are said to {\it{coincide}} if there are unitary operators $u: \mathcal{D}_T \to \mathcal{D}_{T'}$ and $u_{*}: \mathcal{D}_{T^*} \to \mathcal{D}_{{T'}^*}$ such that the following diagram commutes for all $z \in \mathbb{D}$,
$$
\begin{CD}
\mathcal{D}_T @>\Theta_T(z)>> \mathcal{D}_{T^*}\\
@Vu VV @VVu_{*} V\\
\mathcal{D}_{T'} @>>\Theta_{{T'}}(z)> \mathcal{D}_{{T'}^*}
\end{CD}.
$$
\item[(II)]Let $\mathcal G=\{G_i\in\mathcal B(\mathcal D_{T^*}):1\leq i \leq n\}$ and $\mathcal G'=\{G_i'\in \mathcal B(\mathcal D_{T^{' *}}):1\leq i \leq n\}$. We say that the pairs $(\mathcal G, \Theta_T)$ and $(\mathcal G', \Theta_{T'})$ coincide if $\Theta_T$ and $\Theta_{T'}$ coincide and the unitary $u_*:\mathcal D_{T^*}\to \mathcal D_{T^{' *}}$ involved in the coincidence of $\Theta_T$ and $\Theta_{T'}$ has the following intertwining property:
$$
u_*G_i=G_i'u_* \text{ for each }1\leq i \leq n.
$$
\end{enumerate}
\end{definition}
A contraction on a Hilbert space is called {\em{completely-non-unitary}} (c.n.u.) if it has no reducing subspace on which it is unitary. It is well-known (Chapter VI, \cite{Nagy-Foias}) that two c.n.u. contractions are unitarily equivalent if and only if their characteristic functions coincide. Associated to every pair $(T_1,T_2)$ of commuting contractions, we show in \S \ref{Schaffer} that there is a Hilbert space $\mathcal{F}$, an isometry $\Lambda:\mathcal{D}_{T_1T_2}\to \mathcal{F}$, a projection $P$ and a unitary $U$ in $\mathcal{B}(\mathcal{F})$. This is known from the time of And\^o. We call the tuple $(\mathcal{F},\Lambda,P,U)$ the {\em{And\^o tuple}} for $(T_1,T_2)$, see Definition \ref{Ando tuple} below.
\begin{definition}
Let $(T_1,T_2)$ be a pair of commuting contractions and $T=T_1T_2$. Let $(\mathcal{F}_*,\Gamma,P',U')$ be the And\^o tuple for $(T_1^*,T_2^*)$. Define contractions $G_1$, $G_2$ on $\mathcal{D}_{T^*}$ by
$$(G_1,G_2):=(\Gamma^*P'^\perp U'\Gamma,\Gamma^*U'^*P'\Gamma).$$Then the triple $(G_1,G_2,\Theta_T)$ is called the characteristic triple for $(T_1,T_2)$.
\end{definition}
The following theorem, another main result of this paper, is the motivation behind the above definition.
\begin{theorem}\label{uniqueness-pure}
Let $(T_1,T_2)$ and $(T_1',T_2')$ be two pairs of commuting contractions such that their products $T=T_1T_2$ and $T'=T_1'T_2'$ are pure contractions. Then $(T_1,T_2)$ and $(T_1',T_2')$ are unitarily equivalent if and only if their characteristic triples coincide.
\end{theorem}
We now consider the converse direction. We consider two Hilbert spaces $\mathcal{D}$ and $\mathcal{D}_*$ and an {\em{inner}} function $(\mathcal{D},\mathcal{D}_*,\Theta)$ -- this means that for each $z\in\mathbb D$, $\Theta(z)$ is in $\mathcal{B}(\mathcal{D},\mathcal{D}_*)$ and for almost every $z\in \mathbb{T}$, $\Theta(z)$ is an isometry, see Chapter V of \cite{Nagy-Foias} for more details on inner functions. We answer the following question: {\em{For a given inner function $(\mathcal{D},\mathcal{D}_*,\Theta)$ and a pair of contractions $(G_1,G_2)$ on $\mathcal{D}_*$, when does there exist a pair of commuting contractions with $(G_1,G_2,\Theta)$ as its characteristic triple?}}
\begin{definition}
Let $(\mathcal{D},\mathcal{D}_*,\Theta)$ be an inner function and $(G_1,G_2)$ on $\mathcal{D}_*$ be a pair of contractions. We say that the triple $(G_1,G_2,\Theta)$ is admissible if with
$$
\Phi(z)=G_1^*+zG_2 \text{ and } \Psi(z)=G_2^*+zG_1,
$$the operators $M_\Phi$ and $M_\Psi$ on $H^2(\mathcal{D}_*)$ are contractions, $\mathcal{M}:=\Theta H^2(\mathcal{D})$ is joint $(M_\Phi,M_\Psi)$-invariant and
$$
M_\Phi^*M_\Psi^*|_{\mathcal{M}^\perp}=M_\Psi^*M_\Phi^*|_{\mathcal{M}^\perp}=M_z^*|_{\mathcal{M}^\perp}.
$$
We then say that the triple
$$
P_{\mathcal{M}^\perp}(M_{\Phi},M_{\Psi},M_z)|_{\mathcal{M}^\perp}
$$is the functional model associated with the admissible triple.
\end{definition}
The following theorem was obtained recently in (Corollary 4.2) \cite{D-S-S}.
\begin{theorem}[Das-Sarkar-Sarkar, \cite{D-S-S}]\label{DSS-compression}
  Let $(T_1,T_2)$ be a commuting contractions with $T=T_1T_2$ being pure. Then $(T_1,T_2,T)$ is unitarily equivalent to the functional model of its characteristic triple.
\end{theorem}
In \S \ref{FuncModelSec}, we derive Theorem \ref{DSS-compression} as a corollary to Theorem \ref{the2ndmainresult}. \S \ref{FuncModelSec} also establishes the following characterization for admissible triples.
\begin{theorem}\label{charc-admiss}
Let $(\mathcal{D},\mathcal{D}_*,\Theta)$ be an inner function and $(G_1,G_2)$ on $\mathcal{D}_*$ be a pair of contractions.  Then the triple $(G_1,G_2,\Theta)$ is admissible if and only if it is the characteristic triple for some pair of commuting contractions with their product being pure. In fact, $(G_1,G_2,\Theta)$ coincides with the characteristic triple of its functional model.
\end{theorem}
Berger, Coburn and Lebow in (Theorem 3.1, \cite{B-C-L}) found a concrete model for $n$-tuples of commuting isometries, which played a basic role in their investigation of structure of the $C^*$-algebra generated by the commuting isometries and Fredholm theory of its elements. We state their result in the particular case when $n=2$.
\begin{theorem}[Berger-Coburn-Lebow, \cite{B-C-L}]\label{B-C-L}
Let $(V_1,V_2)$ be a pair of commuting isometries on a Hilbert space $\mathcal H$. Then there exists a Hilbert subspace $\mathcal H_u$ of $\mathcal H$ such that each $V_j|_{\mathcal H_u}$ is unitary, $\mathcal H_u^\perp$ is unitarily equivalent to $H^2(\mathcal F)$ for some Hilbert space $\mathcal F$ and under the same unitary
$$
(V_1,V_2, V_1V_2)|_{\mathcal H_u^\perp}\text{ is unitarily equivalent to }(M_{P^\perp U+zPU},M_{U^*P+zU^*P^\perp},M_z),
$$for some unitary $U$ and projection $P$ in $\mathcal B(\mathcal F)$.
\end{theorem}Later in \cite{B-D-F}, Bercovici, Douglas and Foias reconsidered this classification problem for commuting isometries and carried the analysis beyond. In an attempt to generalize these classification results, the following was obtained in (Theorem 3.2) \cite{D-S-S}.
\begin{theorem}[Das-Sarkar-Sarkar, \cite{D-S-S}]\label{D-S-S}
Let $(T_1,T_2)$ be a pair of commuting contractions on a Hilbert space $\mathcal H$ such that their product $T=T_1T_2$ is pure. Then there exist a Hilbert space $\mathcal{F}_*$, a unitary $U'$, a projection $P'$ in $\mathcal B(\mathcal{F}_*)$ and a subspace $\mathcal M$ of $H^2(\mathcal{F}_*)$ jointly co-invariant under $(M_{P'^\perp U'+zP'U'},M_{U'^*P'+zU'^*P'^\perp},M_z)$ such that
$$
(T_1,T_2,T) \text{ is unitarily equivalent to } P_{{\mathcal M}}(M_{P'^\perp U'+zP'U'},M_{U'^*P'+zU'^*P'^\perp},M_z)|_{_{{\mathcal M}}}.
$$
\end{theorem}
Both Theorem \ref{B-C-L} and Theorem \ref{D-S-S} are shown in \S \ref{FuncModelSec} to be corollaries to Theorem \ref{the2ndmainresult}. \S \ref{Schaffer} and \S \ref{Douglas} prove Theorem \ref{themainresult} and Theorem \ref{the2ndmainresult}, respectively. \S \ref{uniquenessSec} proves Theorem \ref{uniqueness} from which follows Theorem \ref{compression-uniquenss}. Theorem \ref{uniqueness-pure} and Theorem \ref{charc-admiss} are proved in \S \ref{FuncModelSec}. And finally in \S \ref{ConRem} we explain how this work is inspired by the tetrablock theory -- the work of Bhattacharyya \cite{sir's tetrablock paper} and discuss a few open problems.
%

\section{The Sch\"affer model for And\^o dilation--proof of Theorem \ref{themainresult}}\label{Schaffer}
Associated to a pair of commuting contractions, there is a unitary and an isometry, known from the time of And\^o. We start by defining these operators as they play a vital role in the construction. Let $(T_1,T_2)$ be a pair of commuting contractions on a Hilbert space $\mathcal H$. Denote by $T$ the product $T_1T_2$. Then
\begin{eqnarray*}
D_T^2=I-T_2^*T_2+T_2^*T_2-T_2^*T_1^*T_1T_2=D_{T_2}^2+T_2^*D_{T_1}^2T_2,
\end{eqnarray*}
which shows that the operator $\Lambda :\mathcal D_T \to \mathcal D_{T_1}\oplus \mathcal D_{T_2}$ defined by
\begin{eqnarray}\label{V}
\Lambda D_Th=D_{T_1}T_2h\oplus D_{T_2}h \text{  for all $h\in \mathcal H$},
\end{eqnarray} is an isometry. Also for every pair of commuting contractions $(T_1,T_2)$, we have
\begin{eqnarray*}
D_{T_2}^2+T_2^*D_{T_1}^2T_2=D_{T_1}^2+T_1^*D_{T_2}^2T_1,
\end{eqnarray*} which shows that the operator $U:\text{Ran}\Lambda \to \mathcal D_{T_1}\oplus \mathcal D_{T_2}$ defined by
\begin{eqnarray}\label{TheUnitary}
U\big(D_{T_1}T_2h\oplus D_{T_2}h\big)=D_{T_1}h\oplus D_{T_2}T_1h \text{  for all $h\in \mathcal H$},
\end{eqnarray}is an isometry. Now one can add an infinite dimensional Hilbert space $l^2$ to $\mathcal D_{T_1}\oplus \mathcal D_{T_2}$ if necessary, to extend $U$ as a unitary. We shall denote the extended unitary operator by $U$ itself and
\begin{eqnarray}\label{the aux space}
\mathcal F:=\mathcal D_{T_1}\oplus \mathcal D_{T_2} \text{ or }\mathcal F:=\mathcal D_{T_1}\oplus (\mathcal D_{T_2}\oplus l^2),
\end{eqnarray}
where $l^2$ is the Hilbert space of square summable sequences (when needed). Let $f$ and $g$ be in $\mathcal D_{T_1}$ and $\mathcal D_{T_2}$, respectively. We denote the member $f\oplus g\oplus 0$ of $\mathcal D_{T_1}\oplus (\mathcal D_{T_2}\oplus l^2)$ just by $f\oplus g$. Armed with this unitary $U$ and the isometry $\Lambda$, we proceed to construct the dilation.
\begin{definition}\label{Ando tuple}
For a pair of commuting contractions $(T_1,T_2)$, let the Hilbert space $\mathcal{F}$, the isometry $\Lambda$ and the unitary $U$ be as defined in (\ref{the aux space}), (\ref{V}) and (\ref{TheUnitary}), respectively. Let $P$ denote the projection of $\mathcal{F}$ onto $\mathcal{D}_{T_1}$ (embedded in the canonical way). The tuple $(\mathcal F,\Lambda, P, U)$ is called the And\^o tuple for $(T_1,T_2)$.
\end{definition}
Let $\mathcal{F}, P$ and $U$ be as in the And\^o tuple for $(T_1,T_2)$. Define two bounded operators $E_1$ and $E_2$ on $\mathcal F$ by
\begin{eqnarray}\label{thefundonE}
E_1:=P^\perp U \text{ and }E_2^*:=PU.
\end{eqnarray} Note that $E_1$ and $E_2$ have the following properties:
\begin{eqnarray}\label{thefundon-E}
E_1\Lambda D_Th=E_1\big(D_{T_1}T_2h\oplus D_{T_2}h\big)&=&0\oplus D_{T_2}T_1h  \text{ and }\\
E_2^*\Lambda D_Th=E_2^*\big(D_{T_1}T_2h\oplus D_{T_2}h\big)&=&D_{T_1}h\oplus 0, \text{  for all $h\in \mathcal H$}.
\end{eqnarray}
\begin{lemma}\label{relations-of-E-lem}
Let $\mathcal{F}$ be any Hilbert space and $E_1,E_2$ be in $\mathcal{B}(\mathcal{F})$. Then $$(E_1,E_2)=(P^\perp U,U^*P)$$for some projection $P$ and a unitary $U$ in $\mathcal{B}(\mathcal{F})$ if and only if $E_1$, $E_2$ satisfy
$$
E_1E_2=E_2E_1=0 \text{ and } E_1E_1^*+E_2^*E_2=E_1^*E_1+E_2E_2^*=I_{\mathcal F}.
$$
\end{lemma}
\begin{proof}
The `only if' part is obvious. The virtue of the converse direction was used in the proof of the Berger-Coburn-Lebow model theory for commuting isometries. But a detailed proof is not found in their paper. So, we prove it here.

Note that any two operators satisfying the above conditions turn out to be partial isometries, because we easily have $E_1E_1^*E_1=E_1$ and $E_2E_2^*E_2=E_2$, which is an equivalent characterization for partial isometries. This again is equivalent to all of $E_1E_1^*, E_1^*E_1,E_2E_2^*$ and $E_2^*E_2$ being projections onto $Ran E_1$, $Ran E_1^*$, $Ran E_2$ and $Ran E_2^*$, respectively. Moreover, the given hypotheses imply that
$$Ran E_1 \oplus Ran E_2^*=\mathcal{F}=Ran E_1^* \oplus Ran E_2.$$ By polar decomposition theorem, we have unitaries $U_1:Ran E_1\to Ran E_1^*$ and $U_2:Ran E_2^*\to Ran E_2$ such that $$E_1^*=U_1(E_1E_1^*)^{\frac{1}{2}} \text{ and }E_2=U_2(E_2^*E_2)^{\frac{1}{2}}.$$Denote the projection $E_1E_1^*$ by $P^\perp$ and the unitary
$$U^*:=U_1\oplus U_2:Ran E_1 \oplus Ran E_2^*\to Ran E_1^* \oplus Ran E_2.$$ With these, the operators $E_1$ and $E_2$ are as in (\ref{thefundonE}).
\end{proof}
We are now ready to do the first construction of an And\^o dilation. Note the similarities with the Sch\"affer construction of minimal isometric dilation for a single contraction.

\begin{theorem}[Sch\"affer model]\label{a-dilation}
Let $(T_1,T_2)$ be a commuting pair of contractions on a Hilbert space $\mathcal H$ and $T=T_1T_2$. Define two operators $V_1,V_2:\mathcal H \oplus H^2(\mathcal F)\to \mathcal H \oplus H^2(\mathcal F)$ by
\begin{eqnarray}\label{feature-demi-dilation}
V_1=\left(
     \begin{array}{cc}
      T_1 & 0 \\
      E_2^*\Lambda D_T & M_{E_1+zE_2^*} \\
     \end{array}
   \right)\text{ and } V_2=\left(
     \begin{array}{cc}
      T_2 & 0 \\
      E_1^*\Lambda D_T & M_{E_2+zE_1^*} \\
     \end{array}
   \right).
\end{eqnarray}
The pair $(V_1,V_2)$ is an isometric dilation of $(T_1,T_2)$.
\end{theorem}
\begin{proof}
All we need to show is that $(V_1,V_2)$ is a commuting pair of isometries. We first show that the $(2,1)$-entry in the matrix representation of both $V_1V_2$ and $V_2V_1$ is $\Lambda D_T$, i.e.,
\begin{eqnarray}\label{proof-1}
E_2^*\Lambda D_TT_2+M_{E_1+zE_2^*}E_1^*\Lambda D_T=E_1^*\Lambda D_TT_1+M_{E_2+zE_1^*}E_2^*\Lambda D_T=\Lambda D_T.
\end{eqnarray}
In the following computation for all $h\in \mathcal H$ we use $E_1E_2=0$:
\begin{eqnarray*}
(E_2^*\Lambda D_TT_2+M_{E_1+zE_2^*}E_1^*\Lambda D_T)h&=&E_2^*\Lambda D_TT_2h+E_1E_1^*\Lambda D_Th\\
&=&PU\Lambda D_TT_2h+P^\perp\Lambda D_Th\\
&=&\Lambda D_Th\\
&=&U^*U(D_{T_1}T_2h\oplus D_{T_2}h)\\
&=&U^*(D_{T_1}h\oplus 0)+U^*(0\oplus D_{T_2}T_1h)\\
&=&U^*PU\Lambda D_Th+U^*P^\perp\Lambda D_TT_1h\\
&=&E_2E_2^*\Lambda D_Th+E_1^*\Lambda D_TT_1h\\
&=&(E_1^*\Lambda D_TT_1+M_{E_2+zE_1^*}E_2^*\Lambda D_T)h.
\end{eqnarray*}
Now Lemma \ref{relations-of-E-lem} shows that $M_{E_1+zE_2^*}M_{E_2+zE_1^*}=M_{E_2+zE_1^*}M_{E_1+zE_2^*}=M_z$. This seals the commutativity part.

It remains to show that $V_1$ and $V_2$ are isometries.
A simple matrix computation shows that $V_1$ would be an isometry if and only if the following equalities hold:
\begin{eqnarray}\label{proof-2}
T_1^*T_1+D_T\Lambda^*E_2E_2^*\Lambda D_T=I_{\mathcal H} \text{ and } D_T\Lambda^*E_2M_{E_1+zE_2^*}=0.
\end{eqnarray}
The first equality is true because for every $h,h'\in \mathcal H$,
\begin{eqnarray*}
\langle E_2^*\Lambda D_Th,E_2^*\Lambda D_Th' \rangle=\langle D_{T_1}h\oplus 0,D_{T_1}h'\oplus 0\rangle =\langle D_{T_1}^2h,h' \rangle,
\end{eqnarray*}and the second equality is true because for every $h\in \mathcal H$, $\zeta\in\mathcal F$, $n\geq 0$,
\begin{eqnarray*}
\langle D_T\Lambda^*E_2M_{E_1+zE_2^*} (z^n\otimes \zeta),h \rangle_{\mathcal H}=\langle z^{n+1}\otimes E_2E_2^*(\zeta), \Lambda D_Th \rangle_{H^2\otimes \mathcal F}=0.
\end{eqnarray*} This shows that $V_1$ is an isometry. Similarly $V_2$ would be an isometry if and only if the following equalities hold true:
\begin{eqnarray}\label{proof-3}
T_2^*T_2+D_T\Lambda^*E_1E_1^*\Lambda D_T=I_{\mathcal H} \text{ and } D_T\Lambda^*E_1M_{E_2+zE_1^*}=0.
\end{eqnarray}
Note that for every $h,h'\in \mathcal H$, we have
\begin{eqnarray*}
\langle E_1^*\Lambda D_Th,E_1^*\Lambda D_Th' \rangle=\langle 0\oplus D_{T_2}h,0\oplus D_{T_2}h'\rangle =\langle D_{T_2}^2h,h' \rangle,
\end{eqnarray*} and for every $\zeta\in\mathcal F$, $n\geq 0$, we have
\begin{eqnarray*}
\langle D_T\Lambda^*E_1M_{E_2+zE_1^*} (z^n\otimes\zeta),h \rangle=\langle z^{n+1}\otimes E_1E_1^*(\zeta), \Lambda D_Th \rangle_{H^2\otimes \mathcal F}=0,
\end{eqnarray*}proving that $V_2$ is an isometry too. This completes the proof..
\end{proof}
The proof of Theorem \ref{a-dilation} shows that if we denote the product $V_1V_2$ by $V$, then
\begin{eqnarray}\label{theV}
V=\left(
     \begin{array}{cc}
      T & 0 \\
      \Lambda D_T & M_{z} \\
     \end{array}
   \right).
   \end{eqnarray}
Note that if $(V_1,V_2)$ is an And\^o dilation of $(T_1,T_2)$, then the product $V=V_1V_2$ is an isometric dilation of the product $T=T_1T_2$. How is the dilation $V=V_1V_2$ related to the Sch\"affer's minimal isometric dilation $V_S$ of $T=T_1T_2$? The theorem below answers this question.
\begin{theorem}
Let $(T_1,T_2)$ be a pair of commuting contractions and $(V_1,V_2)$ be the And\^o dilation of $(T_1,T_2)$ constructed in Theorem \ref{a-dilation}. Then there exists an isometry $\Pi_\Lambda:\mathcal H\oplus H^2(\mathcal D_T)\to\mathcal H\oplus H^2(\mathcal F)$ such that
$$\Pi_\Lambda^*V_1V_2\Pi_\Lambda=V_S,$$where $T=T_1T_2$ and $V_S$ is the minimal isometric dilation of $T$ as in (\ref{Schef}).
\end{theorem}
\begin{proof}
With the isometry $\Lambda$ as in the And\^o tuple for $(T_1,T_2)$, define $\Pi_\Lambda$ by
\begin{eqnarray}\label{thePi}
\Pi_\Lambda:=I_{\mathcal H} \oplus (I_{H^2}\otimes \Lambda):\mathcal H\oplus H^2(\mathcal D_T)\to\mathcal H\oplus H^2(\mathcal F).
\end{eqnarray}Now it is easy to see by the matrix representation (\ref{theV}) of the product $V_1V_2$ that the isometry $\Pi_\Lambda$ has the desired property.
\end{proof}
\begin{proof}[{\bf{Proof of Theorem \ref{themainresult}}}]
({\bf{P}}) $\Rightarrow$ ({\bf{A}}): Let $(T_1,T_2)$ be a pair of commuting contractions on a Hilbert space $\mathcal H$ and $T=T_1T_2$. Then note that the And\^o dilation $(V_1,V_2)$ constructed in Theorem \ref{a-dilation} has all the properties described in part ({\bf{A}}).

({\bf{A}})$\Rightarrow$({\bf{S}}): Suppose there exist a Hilbert space $\mathcal F,$ an isometry $\Lambda:\mathcal D_T\to \mathcal F$ and a commuting pair of isometries $(V_1,V_2)$ with the structure as described in part ({\bf{A}}). Let us denote $P^\perp U$ and $U^*P$ by $E_1$ and $E_2$, respectively. Define two bounded operators $F_1$ and $F_2$ on $\mathcal D_T$ by
\begin{eqnarray}\label{thefund}
F_i:=\Lambda^*E_i\Lambda, \text{ for } i=1,2.
\end{eqnarray}
With the isometry  $\Pi_\Lambda=I_{\mathcal H}\oplus(I_{H^2}\otimes \Lambda)$, define two bounded operators on $\mathcal H\oplus H^2(\mathcal D_T)$ by
$$
S_1:=\Pi_\Lambda^*V_1\Pi_\Lambda=\left(
     \begin{array}{cc}
      T_1 & 0 \\
      F_2^*D_T & M_{F_1+zF_2^*} \\
     \end{array}
   \right)\text{ and }S_2:=\Pi_\Lambda^*V_2\Pi_\Lambda=\left(
     \begin{array}{cc}
      T_2 & 0 \\
      F_1^*D_T & M_{F_2+zF_1^*} \\
     \end{array}
   \right).
$$Then note that $S_1$ and $S_2$ are contractions and have all the properties described in part ({\bf{S}}). Note that the pair $(S_1,S_2)$ need not be commuting, in general.

({\bf{S}})$\Rightarrow$({\bf{P}}): This implication is obvious.
\end{proof}
We end this section with the following remark on operators defined in (\ref{thefund}).
\begin{remark}\label{uni-fund}For a commuting pair of contractions $(T_1,T_2)$, we are going to see in \S \ref{uniquenessSec} (Theorem \ref{fundeqnsThm}), that the contraction operators $F_1,F_2$ on $\mathcal D_T$ as defined in (\ref{thefund}) are \textit{uniquely determined} by the triple $(T_1,T_2,T_1T_2)$.
\end{remark}

\section{The Douglas model for And\^o dilation-Proof of Theorem \ref{the2ndmainresult}}\label{Douglas}
In this section we do the second construction of And\^o dilation. It is well-known that an arbitrary family of commuting isometries can always be extended to a family of commuting unitaries. The following result shows that when the family is finite and one of the isometry in the family is the product of the rest of the isometries, then the family can be extended to family of commuting unitaries with additional structure.
\begin{lemma}\label{special-ext}
Let $\underline{V}=(V_1,V_2,\dots, V_n,V)$ be a commuting tuple of isometries on a Hilbert space $\mathcal{H}$ such that $V=V_1V_2\cdots V_n$, then $\underline{V}$ has a unitary extension $\underline{Y}=(Y_1,Y_2,\dots, Y_n,Y)$ such that $Y=Y_1Y_2\cdots Y_n$ is the minimal unitary extension of $V$.
\end{lemma}
\begin{proof}
We use the Berger-Coburn-Lebow model theory for commuting isometries to prove this result. We prove it for the case when $n=2$. The proof for the general case can be done similarly. So let $(V_1,V_2)$ be a pair of commuting isometries on $\mathcal{H}$ and $V=V_1V_2$. By Wold decomposition (\cite{vonN-Wold}, \cite{Wold}), we know that $$\mathcal{H}=H^2(\mathcal{D}_{V^*})\oplus \mathcal{H}_u, \text{ where } \mathcal{H}_u=\bigcap_{n=0}^\infty V^n\mathcal{H},$$ and with respect to this decomposition, $V=M_z\oplus W$, where $M_z$ is the forward shift on $H^2(\mathcal{D}_{V^*})$ and $W=V|_{\mathcal{H}_u}$ is unitary. It can be checked that the spaces $H^2(\mathcal{D}_{V^*})$ and $\mathcal{H}_u$ are reducing for both $V_1$ and $V_2$ and hence by commutativity $V_i=M_{\varphi_i}\oplus W_i$, where for each $i=1,2$, $\varphi_i\in H^\infty(\mathcal{D}_{V^*})$ and $(W_1,W_2)=(V_1,V_2)|_{\mathcal{H}_u}$ is a pair of commuting unitaries such that $W_1W_2=W$. Since $V_1=V_2^*V$, considering the power series expansion of $\varphi_1$ and $\varphi_2$, one easily concludes that $\varphi_1(z)=F_1+zF_2^*$ and $\varphi_2(z)=F_2+zF_1^*$ for some $F_1,F_2$ in $\mathcal{B}(\mathcal{D}_{V^*})$. Now because $(V_1,V_2)$ are commuting pair of isometries, so are $(M_{\varphi_1},M_{\varphi_2})$ and since $M_{\varphi_1}M_{\varphi_2}=M_z$, by Lemma \ref{relations-of-E-lem} we obtain
$$
(V_1,V_2,V_1V_2)=(M_{P^\perp U+zPU},M_{U^*P+zU^*P^\perp},M_z)
$$for some projection $P$ and unitary $U$ in $\mathcal{D}_{V^*}$.
Now define the following two operators on $L^2(\mathcal{D}_{V^*})\oplus\mathcal{H}_u$:
  $$
  (Y_1,Y_2):=(M_{P^\perp U+e^{it}PU}\oplus W_1,M_{U^*P+e^{it}U^*P^\perp}\oplus W_2).
  $$Clearly $\underline{Y}=(Y_1,Y_2)$ is a unitary extension of $\underline{V}=(V_1,V_2)$. Moreover, $Y=Y_1Y_2=M_{e^{it}}\oplus W$ on $L^2(\mathcal{D}_{V^*})\oplus\mathcal{H}_u$ is clearly the minimal unitary extension of $V=V_1V_2=M_z\oplus W$ on $H^2(\mathcal{D}_{V^*})\oplus\mathcal{H}_u$.
\end{proof} The following well-known result by Douglas is omnipresent in operator theory.
\begin{lem}[Douglas Lemma, \cite{Douglas}]
Let $A$ and $B$ be two bounded operators on a Hilbert space $\mathcal{H}$. Then there exists a contraction $C$ such that $A=BC$ if and only if $$AA^*\leq BB^*.$$
\end{lem}See the paper \cite{Douglas} for a general version of the above lemma. Recall from the Introduction that for a contraction $T$, the operator $Q^2$ is the limit of $T^nT^{* n}$ in the strong operator topology, $X^*:\overline{Ran Q}\to\overline{Ran Q}$ is the isometry such that
\begin{eqnarray}\label{theX}
X^*Q=QT^*,
\end{eqnarray} and $W^*$ on $\mathcal{R}\supseteq \overline{Ran Q}$ is the minimal unitary extension of $X^*$. The isometry $M_z\oplus W$ on $H^2(\mathcal D_{T^*})\oplus \mathcal{R}$, which we denoted by $V_{D}$, is a minimal isometric dilation of $T$ and the isometry $\Pi_{D}:\mathcal{H}\to H^2(\mathcal D_{T^*})\oplus \mathcal{R}$ defined by
\begin{eqnarray}\label{Pi_NF}
\Pi_{D}( h)= \mathcal O (h)\oplus Q(h),
\end{eqnarray}
has the following intertwining property
\begin{eqnarray}\label{The NF}
\Pi_{D}T^*=(M_z\oplus W)^*\Pi_{D},
\end{eqnarray}where $\mathcal{O}$ is the observability operator defined in (\ref{the-iso}).

Let us now take the contraction $T$ to be the product of two commuting contractions $T_1$ and $T_2$. In this case, as we are going to see, many more interesting facts hold. For first example, let $h\in\mathcal{H}$, then
\begin{eqnarray*}
  \langle T_1Q^2T_1^*h,h\rangle = \lim\langle T^n(T_1T_1^*){T^*}^nh,h \rangle \leq \lim\langle T^n{T^*}^nh,h\rangle = \langle Q^2h,h\rangle
\end{eqnarray*}which, by Douglas Lemma, implies that there exists a contraction $X_1^*$ such that $X_1^*Q=QT_1^*$. A similar treatment with the other contraction $T_2$ would give us another contraction $X_2^*$ such that $X_2^*Q=QT_2^*$. Note that
$$X_1^*X_2^*=X^*,$$where $X^*$ is as in (\ref{theX}). It is clear that $X_1$ and $X_2$ commute. Since $X^*$ is an isometry, both $X_1^*$ and $X_2^*$ are isometries. Because, in general, if $T$ is an isometry such that $T=T_1T_2$ for some commuting contractions $T_1$ and $T_2$, then both of $T_1$ and $T_2$ are isometries. It follows from the following norm equalities which we have seen in \S \ref{Schaffer}:
$$
\|D_{T_1}T_2h\|^2+\|D_{T_2}h\|^2=\|D_Th\|^2=\|D_{T_1}h\|^2+\|D_{T_2}T_1h\|^2 \text{ for all }h \in \mathcal{H}.
$$
Also, note that the same is true if the word `isometry' is replaced by `unitary' because the above equalities hold for every contraction, in particular, for $T_1^*$ and $T_2^*$ also.

By Lemma \ref{special-ext}, we have a commuting unitary extension of $(X_1^*,X_2^*)$, $(W_1^*,W_2^*)$ say, on the same space $\mathcal{R}\supseteq \overline{RanQ}$, where $W^*$, the minimal unitary extension of $X^*$ acts and moreover $W=W_1W_2$.

Let $(\mathcal{F}_*,\Gamma, P',U')$ be the And\^o tuple for $(T_1^*,T_2^*)$. Recall that this means
\begin{eqnarray}\label{the aux space*}
\mathcal{F}_*=\mathcal{D}_{T_1^*}\oplus \mathcal{D}_{T_1^*}\text{ or }\mathcal{D}_{T_1^*}\oplus \mathcal{D}_{T_1^*}\oplus l^2,
\end{eqnarray}
$\Gamma:\mathcal D_{T^*}\to \mathcal{F}_*$ is the isometry
\begin{eqnarray}\label{the-iso-pure}
\Gamma D_{T^*}h=D_{T^*_1}T_2^*h\oplus D_{T^*_2}h \text{  for all $h\in \mathcal H$},
 \end{eqnarray}$P'$ is the orthogonal projection of $\mathcal{F}_*$ onto $\mathcal{D}_{T^*}$ and $U':\mathcal{F}_*\to \mathcal{F}_*$ is a unitary that has the following property
\begin{eqnarray}\label{the-uni-pure}
U'\big(D_{T^*_1}T^*_2h\oplus D_{T^*_2}h\big)=D_{T^*_1}h\oplus D_{T^*_2}T^*_1h \text{  for all $h\in \mathcal H$}.
\end{eqnarray} Now define operators $H_1,H_2$ on $\mathcal{F}_*$ by
\begin{eqnarray}\label{the-demi-fund}
(H_1,H_2):=(P'^\perp U',U'^*P')
\end{eqnarray}We shall prove that the pair $(V^D_1,V^D_2)$ on the space $H^2(\mathcal{F}_*)\oplus \mathcal R$ defined by
\begin{eqnarray}\label{TheFS}
V^D_1:=\left(
     \begin{array}{cc}
      M_{H_1^*+zH_2} & 0 \\
      0 & W_1 \\
     \end{array}
   \right)\text{ and }V^D_2:=\left(
     \begin{array}{cc}
      M_{H_2^*+zH_1} & 0 \\
      0 & W_2 \\
     \end{array}
   \right),
\end{eqnarray} is an And\^o dilation for $(T_1,T_2)$.
We need the following result before we can prove that.
\begin{lemma}\label{tet-fund-rels}
For a pair $(T_1,T_2)$ of commuting contractions on a Hilbert space $\mathcal{H}$ and $T=T_1T_2$, we have
$$
\Gamma D_{T^*}T_1^*=H_1\Gamma D_{T^*}+H_2^*\Gamma D_{T^*}T^* \text{ and }
\Gamma D_{T^*}T_2^*=H_2\Gamma D_{T^*}+H_1^*\Gamma D_{T^*}T^*,
$$where the isometry $\Gamma$ and the contractions $(H_1,H_2)$ are as defined in (\ref{the-iso-pure}) and (\ref{the-demi-fund}), respectively.
\end{lemma}
\begin{proof}
We only establish one of the equalities and leave the other as it can be proved similarly. For all $h\in \mathcal H$,
\begin{eqnarray*}
  H_1\Gamma D_{T^*}h+H_2^*\Gamma D_{T^*}T^*h  &=& P'^\perp U'({D_{T_1^*}}T_2^*h\oplus {D_{T_2^*}}h)+P'U'({D_{T_1^*}}T_2^*T^*h\oplus {D_{T_2^*}}T^*h)\\
  &=&(0\oplus D_{T_2^*}T_1^*h)+(D_{T_1^*}T^*h\oplus 0)\\
  &=&D_{T_1^*}T_2^*T_1^*h\oplus D_{T_2^*}T_1^*h=\Gamma D_{T^*}T_1^*h.
\end{eqnarray*}
\end{proof}
\begin{theorem}\label{thedemi2ndmain}
Let $(T_1,T_2)$ be a pair of commuting contractions on a Hilbert space $\mathcal{H}$. Then the pair $(V^D_1,V^D_2)$ on $H^2(\mathcal{F}_*)\oplus \mathcal{R}$ as defined in (\ref{TheFS}) is an isometric dilation of $(T_1,T_2)$.
\end{theorem}
\begin{proof}
Let $T=T_1T_2$. Define the operator $\tilde \Pi:\mathcal{H}\to H^2(\mathcal{F}_*)\oplus \mathcal{R}$ by  $$
  \tilde \Pi (h):= \Pi_\Gamma\Pi_{D}(h)=\sum_{n=0}^{\infty}z^n\Gamma D_{T^*}T^{* n}h\oplus Qh,
  $$ where for $\Gamma$ as defined in (\ref{the-iso-pure}), $\Pi_\Gamma$ is the isometry defined by
\begin{eqnarray}\label{theGamma}
\Pi_\Gamma:=(I_{H^2}\otimes \Gamma)\oplus I_{\mathcal{R}}
           \end{eqnarray} and $\Pi_{D}$ is the isometry as defined in (\ref{Pi_NF}).
  To complete the proof of the theorem we show that
  $$
  (V^D_1,V^D_2)^*\tilde \Pi=\tilde \Pi(T_1,T_2)^*.
  $$
   Lemma \ref{tet-fund-rels} will now come handy to establish these two equalities. We only establish one and the other equality can be established similarly. For $h\in\mathcal{H}$,
\begin{eqnarray*}
  V_1^{D *}\tilde{\Pi}h &=& (M_{H_1^*+zH_2}^*\oplus W_1^*)\tilde{\Pi}h\\
  &=&\left(\sum_{n=0}^{\infty}z^nH_1\Gamma D_{T^*}T^{* n}h+\sum_{n=1}^{\infty}z^{n-1}H_2^*\Gamma D_{T^*}T^{* n}h\right)\oplus W_1^*Qh\\
  &=& \sum_{n=0}^{\infty}z^n\left(H_1\Gamma D_{T^*}+H_2^*\Gamma D_{T^*}T^*\right)T^{* n}h\oplus X_1^*Qh\\
  &=&\sum_{n=0}^{\infty}z^n\Gamma D_{T^*}T^{* n}T_1^*h\oplus QT_1^*h=\tilde{\Pi}T_1^*h.
\end{eqnarray*}
This completes the proof.
\end{proof}
\begin{remark}
Note that the operators $U_1,U_2$ on $L^2(\mathcal{F}_*)\oplus\mathcal{R}$ defined by
$$
U_1:=\left(
     \begin{array}{cc}
      M_{H_1^*+e^{it}H_2} & 0 \\
      0 & W_1 \\
     \end{array}
   \right)\text{ and }U_2:=\left(
     \begin{array}{cc}
      M_{H_2^*+e^{it}H_1} & 0 \\
      0 & W_2 \\
     \end{array}
   \right)
   $$ are commuting unitary extensions of $V^D_1$ and $V^D_2$, respectively, where $(V^D_1,V^D_2)$ are as in Theorem \ref{thedemi2ndmain}. Hence $(U_1,U_2)$ is an And\^o unitary dilation of $(T_1,T_2)$.
\end{remark}
\begin{proof}[{\bf{Proof of theorem \ref{the2ndmainresult}}}]
$({\bf P})\Leftrightarrow({\bf A'}):$  Note that $({\bf P})\Rightarrow({\bf A'})$ is the subject of Theorem \ref{thedemi2ndmain} with $\mathcal{M}$ being the range of the isometry $\tilde{\Pi}$. The direction $({\bf A'})\Rightarrow ({\bf P})$ is clear.

$({\bf P})\Leftrightarrow({\bf D}):$ For the part $({\bf P})\Rightarrow({\bf D})$, let $(V^D_1,V^D_2)$ be as in Theorem \ref{thedemi2ndmain} and define $(D_1,D_2,D):=\Pi_\Gamma^*(V^D_1,V^D_2,V^D)\Pi_\Gamma$, where $\Pi_\Gamma=(I_{H^2}\otimes \Gamma) \oplus I_{\mathcal{R}}.$
Note that since $\tilde{\Pi}=\Pi_\Gamma\Pi_{D}$ and $(V^D_1,V^D_2,V^D)^*\tilde \Pi=\tilde \Pi(T_1,T_2,T_1T_2)^*$, we have
$$\Pi_{D}(T_1,T_2,T_1T_2)^*=\Pi_\Gamma^*(V^D_1,V^D_2,V^D)^*\Pi_\Gamma\Pi_{D}=(D_1,D_2,D)^*\Pi_{D}.$$Therefore if we choose $\mathcal{M'}$ to be the range of $\Pi_{D}$, then we have the operators $D_1,D_2,D$ satisfying all the properties described in $({\bf{D}})$. And finally, $({\bf D})\Rightarrow({\bf P})$ is obvious.
\end{proof}

\section{Uniqueness - proof of Theorem \ref{compression-uniquenss}}\label{uniquenessSec}
It is a one-variable phenomenon that any two minimal isometric (or unitary) dilations of a contraction are unitarily equivalent. It is, however, known \cite{Li-Ti} that two minimal And\^o dilations need not be unitarily equivalent. Hence we do not expect that the two And\^o dilations constructed here are unitarily equivalent. However, equations (\ref{minimality}) and (\ref{NF-min}) reflect a certain beauty in the dilation pairs, viz., the products of the dilation pairs, when compressed to the corresponding minimal isometric dilation spaces of the product $T=T_1T_2$, give us back the minimal isometric dilations of $T=T_1T_2$. Surprisingly, this is good enough for the following two triples of contractions to be unitarily equivalent:
\begin{eqnarray*}
(S_1,S_2,V_S)&:=&\Pi_\Lambda^*(V_1,V_2,V_1V_2)\Pi_\Lambda \text{ and}\\
(D_1,D_2,V_D)&:=&\Pi_\Gamma^*(V_1,V_2,V_1V_2)\Pi_\Gamma,
\end{eqnarray*}
where $(V_1,V_2)$ and $(V^D_1,V^D_2)$ are the And\^o dilations as in Theorem \ref{a-dilation} and Theorem \ref{thedemi2ndmain}, respectively and the isometries $\Pi_\Lambda$ and $\Pi_\Gamma$
are as defined in (\ref{thePi}) and (\ref{theGamma}), respectively. We actually prove a stronger version of Theorem \ref{compression-uniquenss}.

For a pair $\underline{T}:=(T_1,T_2)$ of commuting contractions, let
\begin{eqnarray*}
\mathcal{U}_{\underline{T}}&:=&\{(W_1,W_2,W): \text{ $(W_1,W_2,W)$ is a dilation of $(T_1,T_2,T_1T_2)$},\\&& \text{$W$ is the minimal isometric dilation of $T=T_1T_2$,}\\&& \text{$(W_1,W)$, $(W_2,W)$ are commuting and $W_1=W_2^*W.$}\}
\end{eqnarray*}It can be checked easily that the triples $(S_1,S_2,V_S)$ on $\mathcal{H}\oplus H^2(\mathcal{D}_T)$ and $(D_1,D_2,V_D)$ on $H^2(\mathcal{D}_{T^*})\oplus \mathcal{R}$ are in $\mathcal{U}_{\underline{T}}$.
The following is the main result of this section.
\begin{theorem}\label{uniqueness}
  For a pair $\underline{T}:=(T_1,T_2)$ of commuting contractions, the family $\mathcal{U}_{\underline{T}}$ is a singleton set under unitary equivalence.
\end{theorem}
For a commuting pair $\underline{T}=(T_1,T_2)$ of contractions, let $\underline{W}=(W_1,W_2,W)$ on $\mathcal{K}$ and $\underline{W'}=(W'_1,W_2',W')$ on $\mathcal{K'}$ be in $\mathcal{U}_{\underline{T}}$. Since both of these triples are dilations of $(T_1,T_2,T)$, there exist isometries $\Pi:\mathcal{H}\to\mathcal{K}$ and $\Pi':\mathcal{H}\to\mathcal{K'}$ such that $(T_1,T_2,T_1T_2)\Pi^*=\Pi^*\underline{W}$ and $(T_1,T_2,T_1T_2)\Pi'^*=\Pi'^*\underline{W'}$. Theorem \ref{uniqueness} says that $\underline{W}$ and $\underline{W'}$ are unitarily equivalent, i.e., there exists a unitary $U:\mathcal{K}\to\mathcal{K'}$ such that
$$U\underline{W}=\underline{W'}U \text{ and } U\Pi=\Pi'.$$ In the proof of the above uniqueness theorem we use the following result, which is interesting in its own right, so we state it as a theorem.
\begin{theorem}\label{fundeqnsThm}
  Let $(T_1,T_2)$ be a commuting pair of contractions, $T=T_1T_2$ and $(\mathcal{F},\Lambda, P, U)$ be the And\^o tuple for $(T_1,T_2)$. Define two operators $F_1$, $F_2$ on $\mathcal{D}_T$ by
  \begin{eqnarray}\label{thefunds}
(F_1,F_2):=(\Lambda^*P^\perp U\Lambda,\Lambda^*U^*P\Lambda).
  \end{eqnarray}Then the contractions $F_1$ and $F_2$ satisfy
  \begin{eqnarray}\label{fundeqns}
T_1-T_2^*T=D_TF_1D_T \text{ and }T_2-T_1^*T=D_TF_2D_T.
  \end{eqnarray}
Conversely, any two bounded operators $F_1$, $F_2$ in $\mathcal{B}(\mathcal{D}_T)$ satisfying equations (\ref{fundeqns}) are of the form (\ref{thefunds}), where $(\mathcal{F},\Lambda, P, U)$ is the And\^o tuple for $(T_1,T_2)$.
\end{theorem}
\begin{proof}
Let us call the partial isometries $P^\perp U$ and $U^*P$ by $E_1$ and $E_2$, respectively. Therefore $E_1$ and $E_2$ have the properties (\ref{thefundon-E}). Therefore for every $h,h'\in \mathcal H$ and $i=1,2$,
\begin{eqnarray*}
\langle D_TF_iD_Th,h'\rangle&=&\langle E_i\Lambda D_Th,\Lambda D_Th'\rangle\\&=&\langle E_i\big(D_{T_1}T_2h\oplus D_{T_2}h\big),\big(D_{T_1}T_2h'\oplus D_{T_2}h'\big) \rangle\\
   &=&\left\{
	\begin{array}{ll}
		\langle \big(0\oplus D_{T_2}T_1h\big), \big(D_{T_1}T_2h'\oplus D_{T_2}h'\big)\rangle =\langle (T_1-T_2^*T)h,h' \rangle& \mbox{if } i=1 \\
		\langle \big(D_{T_1}T_2h\oplus D_{T_2}h\big), \big(D_{T_1}h'\oplus 0\big)\rangle =\langle (T_2-T_1^*T)h,h' \rangle&\mbox{if } i=2.
	\end{array}
\right.
\end{eqnarray*}
This proves the first part of the theorem. To prove the second part, let there be two pairs $(F_1,F_2)$ and $(F_1',F_2')$ of operators on $\mathcal{D}_T$ which satisfy equations (\ref{fundeqns}), then $D_T(F_i-F_i')D_T=0$ for $i=1,2$. Since the operators act on the defect space $\mathcal{D}_T$, this implies that $F_i=F_i'$ for $i=1,2$.
\end{proof}
\begin{remark}\label{fundeqnsThm-Rmk}
Note that since Theorem \ref{fundeqnsThm} holds for any pair of commuting contractions, in particular, it holds for the adjoint $(T_1^*,T_2^*)$ also. Therefore we have a result similar to Theorem \ref{fundeqnsThm} in terms of the And\^o tuple $(\mathcal{F}_*,\Gamma,P',U')$ for $(T_1^*,T_2^*)$.
\end{remark}
We are now ready to prove the main result of this section.
\begin{proof}[Proof of Theorem \ref{uniqueness}]
Let $(W_1,W_2,W)$ be in $\mathcal{U}_{\underline{T}}$. We show that $(W_1,W_2,W)$ is unitarily equivalent to $(S_1,S_2,V_S)$. Without loss of generality, we assume that
$$W=\left(
     \begin{array}{cc}
      T & 0 \\
      D_T & M_{z} \\
     \end{array}
   \right)=V_S.$$Since $(W_1,W_2,W)$ is a dilation of $(T_1,T_2,T_1T_2)$, we suppose that with respect to the decomposition $\mathcal{H}\oplus H^2(\mathcal{D}_T)$, $W_1$ and $W_2$ are given by
$$
W_1=\left(
     \begin{array}{cc}
      T_1 & 0 \\
      X_1 & Y_1 \\
     \end{array}
   \right)\text{ and }
   W_2=\left(
     \begin{array}{cc}
      T_2 & 0 \\
      X_2 & Y_2 \\
     \end{array}
   \right),
$$respectively. Since the pairs $(W_1,W)$ and $(W_2,W)$ are commuting, we have by comparing the $(2,2)$ entries, $$Y_iM_z=M_zY_i, \text{ for both $i=1,2$.}$$Hence $Y_i=M_{\varphi_i}$ for some $\varphi_i\in H^\infty(\mathcal{B}(\mathcal D_T))$ for both $i=1,2$. Now $W_1=W_2^*W$ implies the following three equalities:
\begin{eqnarray}\label{1*2}
\begin{cases}
      (a)M_{\varphi_1}=M_{\varphi_2}^*M_z \\
      (b)X_1=M_{\varphi_2}^*D_T\\
      (c)T_1-T_2^*T=D_TF_1D_T.
   \end{cases}
\end{eqnarray}
Considering the power series expansions of $\varphi_1$ and $\varphi_2$ we have by (\ref{1*2})$(a)$,
$$
\varphi_1(z)=F_1+zF_2^* \text{ and }\varphi_2(z)=F_2+zF_1^*
$$for some $F_1$ and $F_2$ in $\mathcal{B}(\mathcal D_T)$. This and (\ref{1*2})$(b)$ imply $X_1=F_2^*D_T$. Now the pair $(W_2,W)$ is commuting and $W$ is an isometry imply $W_2=W_1^*W$, which implies
\begin{eqnarray}\label{2*1}
\begin{cases}
      (a)X_2=F_1^*D_T\\
      (b)T_2-T_1^*T=D_TF_2D_T.
   \end{cases}
\end{eqnarray}
Therefore, we have, so far shown that
$$
W_1=\left(
     \begin{array}{cc}
      T_1 & 0 \\
      F_2^*D_T & M_{F_1+zF_2^*} \\
     \end{array}
   \right)\text{ and }
   W_2=\left(
     \begin{array}{cc}
      T_2 & 0 \\
      F_1^*D_T & M_{F_2+zF_1^*} \\
     \end{array}
   \right).
   $$Note that since $M_{\varphi_1}$ and $M_{\varphi_2}$ are contractions, so are $F_1$ and $F_2$. The rest of the proof follows from equations (\ref{1*2})$(c)$, (\ref{2*1})$(b)$ and Theorem \ref{fundeqnsThm}.
\end{proof}
We have the following direct consequence of Theorem \ref{uniqueness}.
\begin{corollary}
For a pair $(T_1,T_2)$ of commuting contractions and $T=T_1T_2$, if the isometries $\Lambda:\mathcal{D}_T\to \mathcal{F}$ and $\Gamma:\mathcal{D}_{T^*}\to\mathcal{F}_*$ as defined in (\ref{V}) and (\ref{the-iso-pure}), respectively, are surjective, then the And\^o dilations constructed in Theorem \ref{a-dilation} and Theorem \ref{thedemi2ndmain} are unitarily equivalent. In Sz.-Nagy--Foais terminology, if both the factorizations $T=T_1T_2$ and $T^*=T_1^*T_2^*$ are regular (\S 3, Chapter VII, \cite{Nagy-Foias}), then the two And\^o dilations constructed in this paper are unitarily equivalent.
\end{corollary}

\section{Functional models and unitary invariants for the pure case - proof of Theorem \ref{uniqueness-pure} and Theorem \ref{charc-admiss}}\label{FuncModelSec}
The motivation behind this section is the celebrated Sz.-Nagy and Foias model theory for contractions, see Chapter VI of the classic \cite{Nagy-Foias}. The objective of this section is to develop a similar model theory for pairs $(T_1,T_2)$ of commuting contractions such that $T_1T_2$ is pure. First we record the following three consequences of Theorem \ref{the2ndmainresult}.
\begin{proof}[{\bf{Proof of Theorem \ref{B-C-L}}}]
We first compute the And\^o tuple $(\mathcal{V},\Gamma_\nu,P_\nu,U_\nu)$ for $(V_1^*,V_2^*)$, where $(V_1,V_2)$ is a commuting pair of isometries on $\mathcal{H}$. For that we note the following simple fact.
\begin{lemma}\label{justlikethat} For a pair $(V_1,V_2)$ of commuting isometries on a Hilbert space $\mathcal H$,$$\{D_{V^*_1}V_2^*h\oplus D_{V^*_2}h: h\in\mathcal H \}=\mathcal D_{V_1^*}\oplus\mathcal D_{V_2^*}=\{D_{V^*_1}h\oplus D_{V^*_2}V^*_1h: h \in \mathcal H\}.$$
\end{lemma}
\begin{proof}
We only establish the first equality, the proof of the second equality is similar. In the proof we use the basic fact that if $V$ is an isometry, then $D_{V^*}$ is the projection onto ${Ran V_2}^\perp$. Let $f\oplus g \in \mathcal D_{V_1^*}\oplus\mathcal D_{V_2^*}$ be such that $$\langle D_{V^*_1}V_2^*h\oplus D_{V^*_2}h, f\oplus g \rangle=0 \text{ for all $h\in\mathcal H$}.$$ This is equivalent to $\langle D_{V^*_1}V_2^*h, f\rangle + \langle D_{V^*_2}h, g\rangle =0 \text{ for all $h\in\mathcal H$}$, which implies that $\langle V_2^*h, f\rangle + \langle h, g\rangle =0 \text{ for all $h\in\mathcal H$}$. Hence $g=-V_2f$, which implies that $g=D_{V_2^*}g=-(I-V_2V_2^*)V_2f=0. \text{ Hence } f=0 \text{ too}$.
\end{proof}
Denote $\mathcal{V}:=\mathcal D_{V_1^*}\oplus\mathcal D_{V_2^*}$ and $V=V_1V_2$. By Lemma \ref{justlikethat}, we have the operators $\Gamma_\nu:\mathcal D_{V^*}\to\mathcal V$ and $U_\nu:\mathcal V\to \mathcal V$ defined by
\begin{eqnarray}\label{Ando-tuple-isometries}
\;\;\;\;\;\Gamma_\nu D_{V^*}h=D_{V^*_1}V_2^*h\oplus D_{V^*_2}h \text{ and }U_\nu\big(D_{V^*_1}V^*_2h\oplus D_{V^*_2}h\big)=D_{V^*_1}h\oplus D_{V^*_2}V^*_1h,
\end{eqnarray}respectively, are unitaries. In terms of Sz.-Nagy--Foias terminology, this means that the factorization of a co-isometry into the product of contraction is always regular, see (\S 3 in Chapter VII of \cite{Nagy-Foias}).

Now note that when the product $T_1T_2$ is a pure contraction, then the Hilbert space $\mathcal{R}$ in Theorem \ref{the2ndmainresult} is zero. Theorefore applying Theorem \ref{the2ndmainresult} to the pair $(V_1,V_2)$ such that $V_1V_2$ is pure, we get Theorem \ref{B-C-L}.
\end{proof}
 \begin{remark}
Note that our method of the proof of Theorem \ref{B-C-L} reveals that the space $\mathcal F$ in the statement can also be chosen to be $\mathcal D_{V_1^*}\oplus\mathcal D_{V_2^*}$.
\end{remark}
\begin{proof}[{\bf{Proof of Theorem \ref{D-S-S}}}]
It follows from  proof of $({\bf{P}})\Leftrightarrow({\bf{A'}})$ in Theorem \ref{the2ndmainresult} and the fact that $\mathcal{R}=0$.
\end{proof}
\begin{proof}[{\bf{Proof of Theorem \ref{DSS-compression}}}]
It is known from the time of Sz.-Nagy--Foias that when $T$ is a pure contraction, $\Theta_T$ is an inner function and
\begin{eqnarray}\label{represenQ}
\mathcal Q:=Ran \mathcal{O}=(\Theta_{T}H^2(\mathcal D_T))^\perp=H^2(\mathcal D_{T^*})\ominus\Theta_{T}H^2(\mathcal D_T).
\end{eqnarray} Now if we specialize Theorem \ref{the2ndmainresult} to the case when the product $T=T_1T_2$ is pure, then the space $\mathcal{R}=0$ and hence it follows from the proof of the implication $({\bf{P}})\Rightarrow ({\bf{D}})$ that if $(G_1,G_2,\Theta_T)$ is the characteristic triple for $(T_1,T_2)$, then
$$\mathcal{O}(T_1,T_2,T_1T_2)^* =(M_{G_1^*+zG_2},M_{G_2^*+zG_1},M_z)^*\mathcal{O}.$$Therefore $(G_1,G_2,\Theta_T)$ satisfies all the conditions to be an admissible triple and
$$(T_1,T_2,T_1T_2) \text{ is unitarily equivalent to }P_{\mathcal Q}(M_{G_1^*+zG_2},M_{G_2^*+zG_1},M_z)|_{\mathcal Q}.$$
\end{proof}
We shall now prove Theorem \ref{uniqueness-pure} and Theorem \ref{charc-admiss}. The following result will be used.
\begin{theorem}[Sz.-Nagy and Foias, \cite{Nagy-Foias}]\label{NFFuncModel}
If $(\mathcal{D},\mathcal{D}_*,\Theta)$ is an inner function, then
$$
P_{(\Theta H^2(\mathcal{D}))^\perp}M_z|_{(\Theta H^2(\mathcal{D}))^\perp}
$$is a pure contraction with its characteristic function coinciding with $\Theta$.
\end{theorem}
%
\begin{proof}[{\bf{Proof of Theorem \ref{uniqueness-pure}}}]
We first prove the `only if' direction. Let $(T_1,T_2)$ on $\mathcal{H}$ and $(T_1',T_2')$ on $\mathcal{H'}$ be two pairs of commuting contractions such that their products $T=T_1T_2$ and $T=T_1'T_2'$ are pure. Let $(G_1,G_2,\Theta_T)$ and $(G_1',G_2',\Theta_{T'})$ be their characteristic triples. Let $U:\mathcal{H}\to\mathcal{H}'$ be a unitary such that
$
U(T_1,T_2)=(T_1',T_2')U.
$ It can be checked easily that this unitary intertwines the defect operators:
$$UD_T=D_{T'}U\text{ and }UD_{T^*}=D_{T'^*}U.$$ Let $u$ and $u_*$ be the unitaries $u:=U|_{\mathcal{D}_T}$ and $u_*:=U|_{\mathcal{D}_{T^*}}$. For all $h\in\mathcal{H}$,
\begin{eqnarray*}
\Theta_{T'}uD_Th&=&(-T'+zD_{T'^*}(I_{\mathcal{H}'}-zT'^*)^{-1}D_{T'})uD_Th\\
&=&u_*(-T+zD_{T^*}(I_{\mathcal{H}}-zT^*)^{-1}D_{T})D_Th=u_*\Theta_{T}D_Th,
\end{eqnarray*}and by Remark \ref{fundeqnsThm-Rmk},
\begin{eqnarray*}
D_{T^*}u_*^*G_1'u_*D_{T^*}h=u_*^*D_{T'^*}G_1'D_{T'^*}u_*=U^*(T_1'^*-T_2'T'^*)U=T_1^*-T_2T^*=D_{T^*}G_1D_{T^*}
\end{eqnarray*}which by uniqueness gives $u_*^*G_1'u_*=G_1$. Similar computation gives $u_*^*G_2'u_*=G_2$.

Conversely, suppose that the characteristic triples $(G_1,G_2,\Theta_T)$ and $(G_1',G_2',\Theta_{T'})$ of two pairs $(T_1,T_2)$ on $\mathcal{H}$ and $(T_1',T_2')$ on $\mathcal{H'}$ of commuting contractions with their products $T=T_1T_2$ and $T=T_1'T_2'$ being pure, coincide. Let $u:\mathcal{D}_T\to\mathcal{D}_{T'}$ and $u_*:\mathcal{D}_{T^*}\to\mathcal{D}_{T'^*}$ be two unitaries such that $\Theta_{T'}u=u_*\Theta_{T}$ and $u_*G_i=G_i'u_*$, $i=1,2$. It is a matter of straightforward computation to check that the unitary $U_*:=I_{H^2}\otimes u_*$ takes $\Theta_TH^2(\mathcal{D}_T)$ onto $\Theta_{T'}H^2(\mathcal{D}_{T'})$ and intertwines the functional models corresponding to $(G_1,G_2,\Theta_T)$ and $(G_1',G_2',\Theta_{T'})$. Hence an application of Theorem \ref{DSS-compression} seals the deal.
\end{proof}

\begin{proof}[{\bf{Proof of Theorem \ref{charc-admiss}}}]
Let $(\mathcal{D},\mathcal{D}_*,\Theta)$ be an inner function and $G_1,G_2$ on $\mathcal{D}_*$ be contractions such that the triple $(G_1,G_2,\Theta)$ is admissible. Let us define
$$
(T_1,T_2):=P_{\mathcal M^\perp}(M_{G_1^*+zG_2},M_{G_2^*+zG_1})|_{\mathcal M^\perp}, \text{ where }\mathcal{M}:=\Theta H^2(\mathcal{D}).
$$Defining criteria for admissibility imply that $\underline{T}=(T_1,T_2)$ is a commuting pair of contractions on $\mathcal{M}^\perp$ and that $T=T_1T_2$ is a pure contraction. Let $(G_1',G_2',\Theta_T)$ be the characteristic triple for $(T_1,T_2)$. By Theorem \ref{NFFuncModel}, $\Theta$ coincides with $\Theta_T$. This means that there exist unitaries $u:\mathcal{D}\to\mathcal{D}_T$ and $u_*:\mathcal{D}_*\to\mathcal{D}_{T^*}$ such that $\Theta_Tu=u_*\Theta$. Note that since $\Theta$ and $\Theta_T$ are inner functions, both the triples $\underline{G}:=(M_{G_1^*+zG_2},M_{G_2^*+zG_1},M_z)$ on $H^2(\mathcal{D}_*)$ and $\underline{G'}:=(M_{G_1'^*+zG_2'},M_{{G'_2}^*+zG'_1},M_z)$ on $H^2(\mathcal{D}_{T^*})$ are dilations of $(T_1,T_2,T)$ and are in the family $\mathcal{U}_{\underline{T}}$. Hence there exists a unitary $U_*:H^2(\mathcal{D}_*)\to H^2(\mathcal{D}_{T^*})$ such that $U_*\underline{G}=\underline{G'}U_*$ and $U_*|_{(\Theta H^2(\mathcal{D}))^\perp}=I$, where $(\Theta H^2(\mathcal{D}))^\perp$ and $(\Theta_TH^2(\mathcal{D}_T))^\perp$ are identified as
$$
(I-\Theta\Theta^*)f\mapsto(I-\Theta_T\Theta_T^*)(I_{H^2}\otimes u_*)f, \text{ for all }f\in H^2(\mathcal{D}_*).
$$Since both the dilations $M_z$ on $H^2(\mathcal{D}_*)$ and $M_z$ on $H^2(\mathcal{D}_{T^*})$ of $T$ are minimal such a unitary is unique and since $(I_{H^2}\otimes u_*)$ is one such unitary, we conclude that $U_*=I_{H^2}\otimes u_*$. Therefore the two triples $(G_1,G_2,\Theta)$ and $(G_1',G_2',\Theta_T)$ coincide.

Conversely, suppose $(G_1,G_2,\Theta)$ is characteristic triple for some pair $(T_1,T_2)$ of commuting contractions such that $T=T_1T_2$ is pure. In the proof of Theorem \ref{DSS-compression} we observed that every characteristic triple is actually an admissible triple.
\end{proof}
\begin{remark}
For a tuple $\underline{A}=(A_1,A_2,\dots, A_n)$ of operators on $\mathcal{H}$, a tuple $\underline{B}=(B_1,B_2,\dots,B_n)$ of operators acting on $\mathcal{K}$ containing $\mathcal{H}$ is called a {\em{ joint Halmos dilation}} of $\underline{A}$, if there exists an isometry $\Gamma:\mathcal{H}\to\mathcal{K}$ such that $A_i=\Gamma^*B_i\Gamma$ for each $i=1,2,\dots,n$. Note that Theorem \ref{charc-admiss} shows that if $(G_1,G_2,\Theta)$ is an admissible triple, then $(G_1,G_2)$ has a joint Halmos dilation to a commuting pair $(P^\perp U,U^*P)$ of partial isometries, where $P$ is a projection and $U$ is a unitary.

\end{remark}

\section{Concluding Remarks}\label{ConRem}
\subsection{The connection with the tetrablock theory}\label{TetrablockSec}
The purpose of this subsection is to present how the present work relates to the operator theory of a {\it{tetrablock}} domain, which is the following non-convex but polynomially convex domain in $\mathbb C^3$:
$$
\mathbb E=\left\{(x_{11},x_{22},\text{det}X): X=\begin{pmatrix} x_{11} & x_{12} \\ x_{21} & x_{22} \end{pmatrix}\text{ with }\lVert X \rVert <1\right\}.
$$
This domain arose in connection with the $\mu$-synthesis problem that arises in control engineering and was first studied in \cite{awy} for its geometric properties.
The operator theory on the tetrablock was first developed in \cite{sir's tetrablock paper}.
\begin{definition}[Bhattacharyya, \cite{sir's tetrablock paper}]
A triple $(A,B,T)$ of commuting bounded operators on a Hilbert space $\mathcal{H}$ is called a tetrablock contraction if $\overline{\mathbb E}$ is a spectral set for $(A,B,T)$, i.e., the Taylor joint spectrum of $(A,B,T)$ is contained in $\overline{\mathbb E}$ and
$$
||f(A,B,T)|| \leq ||f||_{\infty,\overline{\mathbb E}}=\sup\{ |f(x_1,x_2,x_3)|:(x_1,x_2,x_3) \in \overline{\mathbb E}\}
$$
for any polynomial $f$ in three variables.
\end{definition}
It turns out that because tetrablock is polynomially convex the condition of the Taylor joint spectrum being inside the set, is redundant, see Lemma 3.3 in \cite{sir's tetrablock paper}. Note that a tetrablock contraction $(A,B,T)$ is essentially a triple of commuting contractions, which follows when one chooses $f$ to be the projection polynomials in the definition. The following lemma that led us to the current work, is where the tetrablock contraction theory comes into play in this context.
\begin{lemma}\label{thekeylem}
Let $(T_1,T_2)$ be a commuting pair of contractions on a Hilbert space $\mathcal H$ and $T=T_1T_2$. Then the triple $(T_1,T_2,T)$ is a tetrablock contraction.
\end{lemma}
\begin{proof}
The proof is a simple application of And\^o's Theorem, which in turn proves an analogue of the famous von Neumann inequality \cite{vonN-Wold, Wold} for pairs $(T_1,T_2)$ of commuting contractions acting on Hilbert spaces:
$$
\|f(T_1,T_2)\|\leq \sup\{|f(z_1,z_2)|: (z_1,z_2)\in \overline{\mathbb D^2}\},
$$for every polynomial $f$ in two variables.  Define the map $\pi:\mathbb D\times \mathbb D\to \mathbb C^3$ by
$\pi(z_1,z_2):=(z_1,z_2,z_1z_2).$ Note that Ran$(\pi) \subset  \mathbb E$. Now let $f$ be any polynomial in three variables. By And\^o's theorem,
$$
\|f\circ \pi (T_1,T_2)\|\leq \|f\circ\pi\|_{\infty,  \bar{\mathbb D}^2}\leq \|f\|_{\infty,  \overline{\mathbb E}},
$$ which proves the lemma.
\end{proof}
 Two operators with certain properties play a fundamental role in the study of tetrablock contractions. We need the following notion to describe it. For a bounded operator $F$ on a Hilbert space $\mathcal H$, the {\em{numerical radius}} is defined to be
 $$w(F):=\sup\{|\langle Fh,h\rangle|: h\in \mathcal H\}.$$It was proved in \cite{sir's tetrablock paper} that for every tetrablock contraction $(A,B,T)$ on a Hilbert space $\mathcal H$, there exist two operators $F_1$ and $F_2$ with the numerical radii at most one such that
\begin{eqnarray}\label{fundamental ops}
A-B^*T=D_TF_1D_T \text{ and } B-A^*T=D_TF_2D_T.
\end{eqnarray}It is easy to see that any two operators $F_1$, $F_2$ acting on $\mathcal D_T$ satisfying (\ref{fundamental ops}) are unique. These unique operators are called the fundamental operators of the tetrablock contraction $(A,B,T)$.

Fundamental operators ever since its invention \cite{B-P-SR} have been proved to be of extreme importance in multi-variable dilation theory, see \cite{B-P-SR, sir and me, sir and me1}. The following lemma that follows from Theorem \ref{fundeqnsThm} and the remark following it, shows that the fundamental operators are present in both the above constructions of And\^o dilation also.
\begin{lemma}\label{thefundtheproof}
Let $(T_1,T_2)$ be a pair of commuting contractions and $T=T_1T_2$. Suppose $(\mathcal{F}, \Lambda, P, U)$ and $(\mathcal{F}_*, \Gamma, P', U')$ are the And\^o tuples for $(T_1,T_2)$ and $(T_1^*,T_2^*)$, respectively. Then the pairs of operators $(F_1,F_2)$ on $\mathcal D_T$ and $(G_1,G_2)$ on $\mathcal{D}_{T^*}$ defined by
\begin{eqnarray}\label{tet-fund}
(F_1,F_2):=(\Lambda^* P^\perp U\Lambda,\Lambda^* U^*P\Lambda) \text{ and } (G_1,G_2):=(\Gamma^*P'^\perp U'\Gamma,\Gamma^*U'^*P'\Gamma),
\end{eqnarray}respectively, are the fundamental operators of the tetrablock contractions $(T_1,T_2,T)$ and $(T_1^*,T_2^*,T^*)$, respectively.
\end{lemma}
A normal boundary dilation for the tetrablock was found in \cite{sir's tetrablock paper,sir and me1} under the conditions that the fundamental operators $F_1$, $F_2$ satisfy $[F_1,F_2]=0$ and $[F_1,F_1^*]=[F_2^*,F_2]$. These conditions are not necessarily satisfied by the fundamental operators (as in (\ref{tet-fund})) of $(T_1,T_2,T_1T_2)$, where $(T_1,T_2)$ is a pair of commuting contractions. Therefore the constructions given in \S \ref{Schaffer} and \S \ref{Douglas} are {\it{not}} direct applications of the tetrablock theory.

We now explain why the following model for special tetrablock contractions is a generalized version of both the Berger-Coburn-Lebow and Das-Sarkar-Sarkar models.
\begin{theorem}\label{Sau}[Sau, \cite{sau}]
Let $(A,B,T)$ be a tetrablock contraction on a Hilbert space $\mathcal{H}$ such that $T$ is pure and let $G_1,G_2$ in $\mathcal B(\mathcal D_{T^*})$ be the fundamental operators of $(A^*,B^*,T^*)$. Then $\mathcal{Q}:=H^2(\mathcal{D}_{T^*})\ominus\Theta_TH^2(\mathcal{D}_T)$ is joint $(M_{G_1^*+zG_2},M_{G_2^*+zG_1},M_z)$-invariant and
$$(A,B,T) \text{ is unitarily equivalent to }(P_{\mathcal Q}M_{G_1^*+zG_2}|_{\mathcal Q}, P_{\mathcal Q}M_{G_2^*+zG_1}|_{\mathcal Q},P_{\mathcal Q}M_z|_{\mathcal Q})$$ via the unitary $\mathcal O:\mathcal H\to \mathcal Q$.
\end{theorem}
So, if $(T_1,T_2)$ is a pair of commuting contractions such that $T=T_1T_2$ is pure, then choosing $(A,B,T)=(T_1,T_2,T)$ in Theorem \ref{Sau}, we see that by Lemma \ref{thefundtheproof}, Theorem \ref{D-S-S} and Theorem \ref{DSS-compression} follow.

Since the characteristic function of an isometry is zero, the space $\mathcal Q$ in Theorem \ref{Sau} is $H^2(\mathcal{D}_{T^*})$, hence when $T$ is an isometry
$$(A,B,T) \text{ is unitarily equivalent to }(P_{\mathcal Q}M_{G_1^*+zG_2}|_{\mathcal Q}, P_{\mathcal Q}M_{G_2^*+zG_1}|_{\mathcal Q},P_{\mathcal Q}M_z|_{\mathcal Q}).$$ Now choosing $(A,B,T)=(V_1,V_2,V_1V_2)$ for a pair $(V_1,V_2)$ of commuting isometries, the Berger-Coburn-Lebow model follows from Lemma \ref{justlikethat} and the discussion following it.

\subsection{Future research}\label{FutureResearch}
\begin{enumerate}
\item[(I)]There are at least four different proofs of the classical commutant lifting theorem, see Chapter VII of \cite{FO-FR}. One of the proofs is due to Douglas, Muhly and Pearcy \cite{DMP}. They used the explicit structure of the minimal isometric dilation constructed by Sch\"affer to construct a lifting. A possible application of our explicit constructions of And\^o dilation is the commutant lifting problem for the bidisk: {\em{Given a pair of commuting contractions $(T_1,T_2)$ on a Hilbert space $\mathcal{H}$ with $(V_1,V_2)$ acting on $\mathcal{H}\oplus H^2(\mathcal{F})$ as its And\^o dilation as constructed in Theorem \ref{themainresult} and a bounded operator $X$ (commutant) on $\mathcal{H}$ commuting with $T_1$ and $T_2$, find a necessary and sufficient condition on $X$ for there to exist an operator $Y$ acting on $\mathcal{H}\oplus H^2(\mathcal{F})$ such that $Y$ commutes with $V_1$ and $V_2$, $Y$ is co-extension (lifting) of $X$ with the operator norm $\|X\|$.}} In the case when $\mathcal{H}$ is some reproducing kernel Hilbert space on the bidisk and $T_1,T_2$ are the compressions of the multiplication operators by the co-ordinate functions to a co-invariant subspace, a commutant lifting theorem (for the polydisk in general) was obtained in Theorem 5.1 of \cite{BLT}. An interesting direction for future research would be to consider an arbitrary pair of commuting contractions and construct a lifting of commuting operators using the explicit structure of And\^o dilation $(V_1,V_2)$.
\item[(II)]It will be an interesting future research to find results analogous to Theorem \ref{uniqueness-pure} and Theorem \ref{charc-admiss} for a general pair of commuting contractions.

\item[(III)] The idea of the construction of a tetrablock isometric dilation in \cite{sir's tetrablock paper} is invoked in our first construction of And\^o dilation. Later in \cite{sir and me1} a tetrablock unitary dilation was constructed. On the other hand, if $(U_1,U_2)$ is an And\^o unitary dilation of a pair $(T_1,T_2)$ of commuting contractions, then one easily sees that $(U_1,U_2,U_1U_2)$ is a tetrablock unitary dilation of $(T_1,T_2,T_1T_2)$. But one can check that, unfortunately, a similar use of the ideas invoked in \cite{sir and me1} does not work for a Sch\"affer-type construction of And\^o unitary dilation. So a Sch\"affer-type construction of And\^o unitary dilation still remains open.
\end{enumerate}

\section{Acknowledgement}The author thanks Prof. B. Krishna Das for reading a previous version of this paper and giving helpful suggestions and comments. The author wishes to express gratitude to Prof. T. Bhattacharyya for a short inspiring discussion at the beginning of this project. The author also wishes to profusely thank Prof. Joseph A. Ball for reading previous versions of this paper with immense patience and for suggesting Theorem \ref{charc-admiss}.

\end{document}